\documentclass{amsart}  	
\usepackage{geometry}                		
\usepackage{graphicx}				
\usepackage{amssymb}

\usepackage{amsthm}
\usepackage{graphicx}
\usepackage{mathtools}
\usepackage{amsmath,amsfonts}
\usepackage{epsfig}
\usepackage{comment}
\usepackage{color}
\usepackage{amsmath}

\usepackage{todonotes}

\usepackage[nocompress]{cite}

\newcommand{\Mod}{{\rm Mod}}

\newcommand{\R}{{\mathbb R}}

\newcommand{\N}{{\mathbb N}}

\newcommand{\bfB}{{\mathbf{B}}}
\newcommand{\cK}{{\mathcal K}}

\newcommand{\C}{{\mathbb C}}

\newcommand{\Hdim}{\cH_{\dims}}

\newcommand{\unit}{1}

\newcommand{\loc}{{\mbox{\scriptsize{loc}}}}

\newcommand{\cE}{\mathcal{E}}

\newcommand{\clB}{{\overline{B}}}

\newcommand{\lK}{\overline{K}}

\newcommand{\cH}{\mathcal{H}}
\newcommand{\cB}{\mathcal{B}}
\newcommand{\bfx}{\textbf{x}}
\newcommand{\cC}{\mathcal{C}}

\newcommand{\dims}{{\rm dim}}
\newcommand{\rad}{{\rm rad}}

\newcommand{\defeq}{:=}

\DeclareMathOperator{\diam}{diam}

\def\vint_#1{\mathchoice
          {\mathop{\vrule width 6pt height 3 pt depth -2.5pt
                  \kern -8pt \intop}\nolimits_{#1}}%
          {\mathop{\vrule width 5pt height 3 pt depth -2.6pt
                  \kern -6pt \intop}\nolimits_{#1}}%
          {\mathop{\vrule width 5pt height 3 pt depth -2.6pt
                  \kern -6pt \intop}\nolimits_{#1}}%
          {\mathop{\vrule width 5pt height 3 pt depth -2.6pt
                  \kern -6pt \intop}\nolimits_{#1}}}

\numberwithin{equation}{section}
\theoremstyle{plain}
\newtheorem{theorem}[equation]{Theorem}
\newtheorem{corollary}[equation]{Corollary}
\newtheorem{lemma}[equation]{Lemma}

\theoremstyle{definition}

\newtheorem{remark}[equation]{Remark}

\title{A New Hausdorff Content Bound for Limsup Sets}
\author{Sylvester Eriksson-Bique}
\address{Research Unit of Mathematical Sciences,
P.O.Box 8000,
FI-90014 Oulu, Finland}
\email{\tt sylvester.eriksson-bique@oulu.fi}

\subjclass[2020]{28A78,28A75,30L99 (28A80,49Q15,26B30,31E05)}

\begin{document}
\maketitle

\begin{abstract}
We give a new Hausdorff content bound for limsup sets, which is related to Falconer's sets of large intersection. 
Falconer's sets of large intersection satisfy a content bound  for all balls in a space. In comparison, our main theorem only assumes a scale-invariant bound for the balls forming the limit superior set in question. 

We give four applications of these ideas and our main theorem: a new proof and generalization of the mass transference principle related to Diophantine approximations, a related result on random limsup sets, a new proof of Federer's characterization of sets of finite perimeter and a statement concerning generic paths and the measure theoretic boundary. The new general mass transference principle transfers a content bound of one collection of balls, to the content bound of another collection of sets -- however, this content bound must hold on all balls in the space. The benefit of our approach is greatly simplified arguments as well as new tools to estimate Hausdorff content. 

The new methods allow for us to dispense with many of the assumptions in prior work. Specifically, our general Mass Transference Principle, and bounds on random limsup sets, do not assume Ahlfors regularity. Further, they apply to any complete metric space. This generality is made possible by the fact that our general Hausdorff content estimate applies to limsup sets in any complete metric space.
\end{abstract}

\section{Introduction}

Consider a collection of closed balls $\mathcal{B}$ in a \emph{complete} metric space $X$. One can form two collections: their \emph{union} $\bigcup \mathcal{B}$ consisting of points that belong to at least one of the elements, and their \emph{limsup set} $\limsup \mathcal{B}$ given by

\[
\limsup \mathcal{B} \defeq \{x : x \text{ lies in infinitely many distinct balls } B_i \in \mathcal{B} \text{ with } \lim_{i\to\infty} \rad(B_i)=0\}.
\]
Such sets arise naturally in a variety of fields: Diophantine approximations, dynamics, probability, and, as we will see, in the study of sets of finite perimeter and isoperimetry. 
It is obvious, that $\limsup \mathcal{B} \subset \bigcup \mathcal{B}$ and often the limsup set is much smaller than the union. However, our main theorem shows that, if the union of the sets is large \emph{in a scale invariant way}, then also $\limsup \mathcal{B}$ is substantial. The size is measured by a general Hausdorff content; see Section \ref{sec:prelim} for definitions of the content.  As a consequence of this general principle and techniques, we obtain a number of applications:

\begin{enumerate}
\item \textbf{A new proof and generalization of the Mass Transference Principle of Beresnevich and Velani \cite{beresnevichvelani}.} 
There has appeared quite a substantial number of works building and generalizing the work of Beresnevich and Velani. MathSciNet knew, in 2021, of 97 references to their work; see e.g. \cite{HS, DB, E, Z, AB, KR,WW,WWX}. Many of these are various generalizations. A comprehensive survey is available in \cite{DTsurvey}.  While most of this work assumed some Ahlfors regularity, our theorem does not require this assumption. Further, our approach uses substantially different techniques.
\item \textbf{A new result on the content of random limsup sets.} Our version of the Mass Transference Principle together with new covering lemmas yields strengthened and simplified bounds for the dimension and content of limsup sets of collections of \emph{open} sets. The proofs yield also that, in Euclidean spaces, such limsup sets are sets of large intersection in the sense of Falconer. For prior results, see \cite{EJJ, EP, FJJV}.
\item \textbf{A new proof of a geometric characterization for sets of finite perimeter in PI-spaces.} Sets of finite perimeter in Euclidean spaces are classical, see e.g. \cite{evansgariepy}, and their theory has been extended to general metric measure spaces \cite{A,Adoubling, M} . A natural problem is giving geometric characterizations for a set to be of finite perimeter. The Federer characterization is one of the most natural ones and involves a Hausdorff content bound for the measure theoretic boundary; for the Euclidean version see \cite[Theorem 4.5.11]{F}. Lahti extended Federer's characterization to spaces satisfying doubling and a Poincar\'e inequality \cite{LFederer}. These spaces are called \emph{PI-spaces}. We present a much shorter proof for Lahti's and Federer's claim, which uses less machinery and leads to better bounds on the constants involved. This proof is new even in the Euclidean case.
\item \textbf{A result that generic curves in PI-spaces pass through the measure theoretic boundary.} For $p\in [1,\infty)$, the modulus of a curve family $\Gamma$, denoted by $\Mod_p(\Gamma)$, is a measurement for the size of curve families. Specifically, ($p$-)exceptional curve families $\Gamma$ are those for which $\Mod_p(\Gamma)=0$. Interesting analytic properties hold outside of $p$-exceptional curve families. In other words, they hold for $p$-almost every curve. Our covering lemmas and proofs yield a new proof of the fact that in a PI-space the following hold: If $E \subset X$ is a set of finite perimeter, then $1$-a.e. curve going from the measure theoretic interior to the measure theoretic exterior passes through the measure theoretic boundary. This is a fundamental property, which was first proven in \cite[Corollary 6.4.]{NL}. Indeed, it is closely connected to the Federer Characterization, although our proof does not utilize this fact; see \cite[Theorem 6.5]{NL}.
\end{enumerate}

In this introduction, and paper, we will first discuss our general result on Hausdorff content. Then, we apply it to the contexts stated above. 
 
To state our main theorem we fix some notation. First define the notation for the collections of sets we consider. Open balls in a metric space $X$ will be denoted $B=B(x,r)=\{y\in X: d(x,y) < r \}$ with radius given by $r=\rad(B)$. For us, it will be crucial, that each ball $B$ comes with an associated radius $r$ and center $x$. Thus, formally, a ball would be a triple: $(x,r,B)$. However, we simply write $B=B(x,r)$ to specify the center and radius, and suppress $(x,r)$ in most of the notation.

 Closed balls are defined by $\clB=\clB(x,r) = \{y\in X:d(x,y) \leq r\}$ -- and this is not to be confused with the closure of the open ball. When we simply say "a ball $B$", we mean either an open or closed ball, and use $B$ to denote such a ball. Whenever necessary, we will distinguish between open and closed balls, and we use only $\clB$ to indicate closed balls. 
Given a collection $\mathcal{B}$ of closed balls, a scale $r\in (0,\infty)$ and a set $A\subset X$, we define the scale and location restricted collections by
\begin{align}
\mathcal{B}|_{r}&\defeq\{\clB \in \mathcal{B}: \rad(\clB) \leq r\},  \ \mathcal{B}|^{A}\defeq \{\clB \in \mathcal{B}: \clB \subset A\} \ \text{ and }  \nonumber\\
 \mathcal{B}|_{r}^A &\defeq \{\clB \in \mathcal{B}: \rad(\clB) \leq r, \clB \subset A\}. \label{eq:def_restrictions}
\end{align}
 Given a number $C\geq 1$ the inflation of an open ball $B=B(x,r)$ and a closed ball $\clB=\clB(x,r)$ are denoted by $CB,C\clB$ and are given by $CB\defeq B(x,Cr)$ or $C\overline{B} \defeq \clB(x,Cr)$. The inflation of all balls in a collection is given by $C\mathcal{B}=\{C\clB : \clB \in \mathcal{B}\}.$ All of these notions naturally extend to collections of open balls.

Next, the size of our sets is measured in terms of a spherical Hausdorff content. A \emph{gauge function} $\xi$ associates to each closed ball $\clB=\clB(x,r)$ and each open ball $B=B(x,r)$ a number $\xi(B),\xi(\clB) \in (0,\infty)$ with $\xi(B)=\xi(\clB)$. We will assume that $\xi$ is a \emph{doubling} gauge function: there exists a constant $D$ so that for any two open balls $B,B'$ with $B \subset B'$ and $\rad(B) \leq \rad(B') \leq 2\rad(B)$, we have $D^{-1}\xi(B) \leq \xi(B')\leq D\xi(B)$.\footnote{A more familiar form for this may be $\xi(2B) \leq D\xi(B)$. However, since we do not assume that $\xi(B) \leq \xi(B')$ when $B\subset B'$, then this slightly more technical formulation is needed. Indeed, our application to finite perimeter sets will require us to consider such a gauge function.} If we wish to emphasize the constant, we will call such a function a \emph{$D-$doubling gauge function}. 

Next, we state the main theorem of this paper, which yields a way to transfer bounds of Hasdorff contents of unions to those for limsup sets. The new aspect of this statement is that the content bound is assumed only for the balls in the collection, and that the setting is that of all complete metric spaces.

\begin{theorem}\label{thm:inflation-mainthm}Let $\xi$ be a $(D-)$doubling gauge function and $\Delta,\delta \in (0,1)$ be any constants.  There exists a constant $\beta=\beta(\delta, D)$ so that the following holds.
Suppose that $X$ is a complete metric space and that $\mathcal{B}$ is any collection of closed balls, with the following: For each $\clB_0=\clB(x_0,r_0) \in \mathcal{B}$  and any $r\leq r_0$ we have

\[
\mathcal{H}^\xi_{r_0}\left(\bigcup \, \mathcal{B}|_{r}^{\clB_0}\right) \geq \Delta  \xi(\clB_0).
\]

Then, for each $\clB_0\in \mathcal{B}$ we have

\[
\mathcal{H}^\xi_{r_0}(\limsup \ (1+\delta)\  \mathcal{B}|^{\clB_0})\geq \Delta \beta \xi(\clB_0).
\]
\end{theorem}

In other words, in self-similar settings, it is equally hard to prove a Hausdorff-content for the limsup set, as it is to prove a content bound for a union. This insight may simplify many arguments, where the main technical step is the limit step.

\begin{remark}\label{rmk:assumptions-mainthm} Without assuming completeness, one can find counter examples. For instance, let $X=[0,1] \setminus \{2^{-1}\}$ and $\xi$ be the gauge given by $\xi(B)=1$. Consider $\mathcal{B}=\{\clB((2^{-1}-2^{-1}n^{-1}, 1/n) : n \in \N\}$. In this case $\limsup 2\mathcal{B} = \emptyset$, while $\cH^\xi_{r_0}(\mathcal{B}|^{\clB}_r) \geq 1$ for every $\clB \in \mathcal{B}$. 

The assumption that the collections consist of closed balls is less severe. For simplicity, throughout, for our theorems on Hausdorff content we assume closed balls. However, if $\cB$ is a collection of open balls, which verifies the assumption of Theorem \ref{thm:inflation-mainthm}, then $\overline{\mathcal{B}}=\{\clB(x,Cr) : B(x,r) \in \cB\}$, for any $(1+\delta/4)>C>1$, would also verify the assumption. Thus, with $\delta/4$ replacing $\delta$, we obtain also a lower content bound for $\limsup \cB$ for collections of open balls.
\end{remark}

In the applications that we shall discuss, the minor inflation factor in the statement is harmless. If desired, it could be placed in the assumption. 

\begin{theorem}\label{thm:inflation-mainthm-2}Let $\xi$ be a $(D-)$doubling gauge function. Let $\Delta \in (0,1)$ be any constant.   
 There exists a constant $\beta=\beta(\eta, D)$ so that the following holds.
Suppose that $X$ is a complete metric space and that $\mathcal{B}$ is any collection of closed balls, with the following: For each $\clB_0=\clB(x_0,r_0) \in \mathcal{B}$  and any $r\leq r_0$ we have

\[
\mathcal{H}^\xi_{r_0}\left(\bigcup \, \mathcal{B}|_{r}^{\frac{1}{2}\clB_0}\right) \geq \Delta  \xi(\clB_0).
\]

For each $\clB_0\in \mathcal{B}$ we have

\[
\mathcal{H}^\xi_{r_0}(\limsup \  \mathcal{B}|_{r_0}^{\clB_0})\geq \Delta \beta \xi(\clB_0).
\]
\end{theorem}

The statements of these theorems hide a key, and somewhat elementary, observation, which we now explain. 
The limit superior set of $(1+\delta)\mathcal{B}$ can be expressed as an infinite intersection by 
\[
\limsup \ (1+\delta)\, \mathcal{B}|^{\clB_0}_{r_0} = \bigcap_{n \in\N} \bigcup \ (1+\delta)\  \mathcal{B}|_{r_0/n}^{\clB_0}.
\]
 The assumption of the statement implies that each set $A_n = \bigcup (1+\delta)\mathcal{B}_{r_0/n}^{\clB_0}$ on the right hand side has a content lower bound (in a scale invariant way). The conclusion is that the intersection $\cap A_n$ has also a content lower bound. 

\emph{In general}, it is not possible to pass from information on the Hausdorff content of a nested sequence of sets $A_n$ to the content of the intersection. Indeed, generally this holds only in special cases, of which we mention two: when each $A_n$ is a compact set; and when each $A_n$ is a set of large intersection, as defined by  Falconer in  \cite{falconersets}.  
A simple counter-example for the general statement is given by $A_n=(0,\frac{1}{n})$. The $0$-Hausdorff content satisfies $\mathcal{H}^0_{r}(A_n)\geq 1$ for each $r\in (0,1)$.  However, their intersection is empty, $\cap_{n=1}^\infty A_n = \emptyset$, and has zero content. Note that the $0$-Hausdorff content corresponds to $\xi(B)=1$ for any ball $B$. Other dimensions and gauge functions could be obtained by taking products of this example with other spaces.

The conclusion of our theorem is that a content bound is obtained for the intersection, if the content bound is scale-invariant in the stated way. Let us show how the proof works for the counting measure,  when $\xi(B)=1$ for each ball $B$, which illustrates how the scale-invariance is employed.

\vspace{.1cm}
\noindent \hrulefill
\vspace{.1cm}

\noindent \textbf{Case of a Counting measure:} The assumption states that for every ball $\clB_0 \in \mathcal{B}$ of radius $r_0$, there is at least one ball $\clB_1 \in \mathcal{B}$ with $\clB_1 \subset \clB_0$ with radius $r_1$ satisfying $r_1\leq r_0/2$. One can apply this recursively: Assuming $k\geq 1$ and that $\clB_k$ is defined with radius $r_k \leq r_{k-1}/2$, choose $\clB_{k+1} \in \mathcal{B}$ with $\clB_{k+1}\subset \clB_k$ and with radius $r_{k+1}$ satisfying $r_{k+1}\leq r_k/2$. Then $\cap_{k=1}^\infty \clB_k \neq \emptyset$ since $X$ is complete, and thus $\mathcal{H}^\xi_\infty(\limsup \mathcal{B}|^{\clB_0}) \geq 1$ since $\cap_{k=1}^\infty \clB_k\subset \limsup\mathcal{B}|_{\clB_0}$.

\vspace{.1cm}
\noindent \hrulefill
\vspace{.1cm}

While this argument seems simple, we will see it in the proof of our main theorem as embedded in one of the sub-cases. However, once we use different $\xi$, such as the $k$-dimensional Hausdorff gauge $\xi(B)=\rad(B)^k$ for $k>0$, the assumption doesn't simply force one sub-ball but  \emph{many} balls which form a ``second level''. This idea of levels underlies the second sub-case on the proof of our main theorem. Roughly, the idea is to take an iterated intersection of such levels. The proof is however more delicate, as a simple covering and doubling argument loses a factor which ought to be controlled.

\vspace{.5cm}






\noindent \textbf{Acknowledgements:} We thank Esa Järvenpää and Maarit Järvenpää for extensive discussions on the paper. They pointed out the reference \cite{falconersets}, as well as many other references on the topic. Many results of the paper were also strengthened and improved thanks to their feedback. We also thank Nageswari Shanmugalingam for first posing the problem of proving Federer's characterization, which surprisingly lead to the study of the general Hausdorff content theorems and lemmas presented in this paper.  The author was supported in part by the Finnish Academy grant \# 345005.

\section{Applications of Theorem}
\label{sec:introapp}
The main motivation of this paper is to prove Theorem \ref{thm:inflation-mainthm}. However, to demonstrate its importance, we apply it in different contexts: mass transference principle and diophantine approximations,  structure and characterization of finite perimeter sets and finally, to size bounds for random limsup sets. 

We find it curious, that Theorem \ref{thm:inflation-mainthm} unifies these quite distinct contexts. Further, it is central, that our proofs rely mainly on simple covering arguments stated in Lemmas \ref{lem:5covering} and \ref{lem:nice-collection} - which are completely classical and somewhat elementary. It is  surprising to us, that such basic techniques yield such strong results. Further, we will see, that the simplification of the methods yields sharper and more general conclusions; in particular for the Mass Transference Principle and random Limit superior sets.

The final application below on generic curves, does not directly apply a version of Theorem \ref{thm:inflation-mainthm} in its proof. It is instead proved directly from the covering Lemmas mentioned above. However, we include it here, since its closely connected to our proof of Federer's characterization, and the proof technique is similar in spirit.

\subsection{Mass Transference Principle}

The Mass Transference Principle was proven by Beresnevich and Velani in \cite{beresnevichvelani}. As the name suggests, it is a principle which allows to convert information about the Hausdorff-content of one collection of sets, to the Hausdorff-content of another collection of sets. Originally, both of these collections were sets of balls arising in Diophantine approximations. However, later work considered a variety of different types of sets. An excellent survey is provided in \cite{DTsurvey}. More references were provided above at the beginning of the introduction. 

Before the statement, we introduce some terminology. To match many of the previous references, we will use $k$ to denote the Hausdorff dimension of the space $X$. An increasing function $f:(0,\infty) \to (0,\infty)$  is called a dimension function.  Hausdorff contents and measures $\mathcal{H}^f$  are defined in Section \ref{sec:prelim}. When $f$ is a dimension function, Let $B^f=B(x,f(r)^{1/k})$ for  $B=B(x,r)$.

If $(\mathcal{B}_i)_{i\in \N}$ is a sequence of collections of balls, we define the limit superior set of their unions by
\[
\limsup_{i\to\infty} \mathcal{B}_i = \cap_{N=1}^\infty \cup_{i=N}^\infty \cup \mathcal{B}_i.
\]
This corresponds to the set of points $p$, for which there exists an infinite increasing sequence  $(i_l)_{l\in \N}$ and balls $B_l \in \mathcal{B}_{i_l}$ with $p \in B_l$ for each $l\in \N$. When the collections $\mathcal{B}_i$ consist of a single set, that is $\mathcal{B}_i = \{B_i\}$, we also write $\limsup_{i\to\infty} B_i$.

The classical Mass transference principle from \cite{beresnevichvelani} states the following. 

\begin{theorem}[Theorems 2 and 3 in \cite{beresnevichvelani}]\label{thm:MTP} Suppose that $X$ is complete and Ahlfors $k$-regular.
Let $\{B_i\}_{i\in \N}$ be a sequence of (open) balls with $\lim_{i\to\infty}\rad(B_i) = 0$ for which 

\[
\mathcal{H}^k(\limsup_{i\to\infty} B_i^f \cap B) = \mathcal{H}^k(B),
\]
for each (open) ball $B\subset X$, then for any dimension function $f$ for which $f(x)x^{-k}$ is monotone, we have

\[
\mathcal{H}^f(\limsup_{i\to\infty} B_i \cap B) = \mathcal{H}^f(B).
\]
\end{theorem}

We reprove this result and extend it to consider \emph{general collections of open sets} and \emph{general gauge functions}. In particular, we consider also doubling gauge functions $\xi$ which do not come from a dimension function. An important difference between the settings of Theorem \ref{thm:MTP} and our main Theorem \ref{thm:inflation-mainthm}, is that in the Mass transference principle, we place an assumption on \emph{all} balls in the space. Our generalized mass transference principle will also place such an assumption -- in contrast to Theorem \ref{thm:inflation-mainthm} which only places an assumption on the collection of balls defining the limit superior set.

First, we define limsup sets of collections of open sets. Suppose that $\mathcal{E}_i$ is a nested sequence of collections of sets. Then, we consider the limsup set $\limsup_{i\to \infty} \mathcal{E}_i$ which is given by 

\[
\limsup_{i\to \infty} \mathcal{E}_i \defeq \{x : \forall i \in \N, \exists E_i \in \mathcal{E}_i, x \in E_i\}.
\]
In contrast to the case of balls, we do not impose any condition on the radius or diameter. 


Next, we will consider two gauge functions $\xi$ and $\nu$. Here, $\nu$ plays the role of the $k$-Ahlfors regular measure on $X$ in Theorem \ref{thm:MTP}, whereas $\xi$ plays the role of the dimension function $f$. The assumption $f(x)x^{-k}$ being monotone is translated by the concept of domination. If $\nu$ and $\xi$ are $D$-doubling gauge functions, we say that $\xi$ dominates $\nu$, or $\nu \ll \xi$, if for any two balls $B \subset B'$ with $\rad(B') \leq \rad(B)$ it holds that
\[
\frac{\xi(B')}{\nu(B')} \leq \frac{\xi(B)}{\nu(B)}.
\]
The assumption of domination corresponds directly with the assumption of $f(x)x^{-k}$ being monotone decreasing. Indeed, from here on, we leave out the discussion of the case when $f(x)x^{-k}$ is monotone increasing. That case of Theorem \ref{thm:MTP} is nearly a triviality; see e.g. \cite{beresnevichvelani}.

The Alhfors regularity assumption is replaced by the assumption that $\xi$ and $\nu$ are doubling. Note that, if $\nu$ is doubling and $\nu \ll \xi$, then it is direct to show that $\xi$ is doubling as well. Finally, we need to define the transference collection of balls, which previously was denoted $B^f$.

Suppose that $\xi$ and $\nu$ are two doubling gauge functions. Given a set $E$, we say that a ball $B$ is a $(\xi,\nu)$-transfer ball for $E$ if $\mathcal{H}^\xi_{\rad(B)}(E\cap B) \geq \nu(B)$. We denote this relationship by $E \succ_{\xi,\nu} B$. We say that a collection $\mathcal{B}$ is a $(\xi,\nu)$-transfer collection of balls, if for each $B \in \mathcal{B}$ there exists $E \in \mathcal{E}$ so that $E \succ_{\xi,\nu} B$. We write also $\mathcal{E}\succ_{\xi,\nu} \mathcal{B}$ in this case. Compared to earlier work on transference principles, we do not assume that $E \subset B$; see e.g. \cite{KR}.

We give two examples when $X=\R^k$ and when $\nu(B)=\rad(B)^k$ is given by the (renormalized) Lebesgue measure and $\xi(B)=f(\rad(B))$ for some dimension function. We also assume that $f(x)x^{-k}$ is monotone decreasing. This yields that $\xi$ is doubling.

\begin{enumerate}
\item Since $f(x)x^{-k}$ is monotone decreasing, then $\cH^\xi_{\rad}(B) \geq f(\rad(B))$ for all balls $B$. Thus, for all balls $B$, we have $B \succ_{\xi,\nu} B^f$.
\item Take any bounded open set $E \subset \R^k$, and any $x\in E$. Suppose that $\lim_{r\to 0} f(r)r^{-k} = \infty$. Let $r_0=\sup\{ r>0 : \mathcal{H}^\xi_{\infty}(E \cap B(x,r)) \geq r^{k} \}$. By openness and boundedness of $E$, we have $r_0 \in (0,\infty)$. Moreover, we get that the open ball $B_x = B(x,r_0)$ satisfies $E \succ_{\xi,\nu} B_x$. In particular, $E \subset \bigcup_{x\in E} B_x$. Such a construction will play a crucial role in Lemma \ref{lem:transfer_set_measure} and later in the proof of dimension bounds for random limsup sets.

\end{enumerate}

Our generalized mass transference principle is stated as follows. Here, $\rad(\cB)=\sup_{B\in \cB} \rad(B)$.

\begin{theorem}\label{thm:genmasstransf} Let $X$ be a complete metric space and $\xi,\nu$ be two $D$-doubling gauge functions with $\nu \ll \xi$. For every $\Delta>0$ there exists a $\beta>0$ for which the following holds.

Suppose that $\mathcal{E}_i$ is a sequence of collections of open sets with $\mathcal{E}_i \subset \mathcal{E}_j$ for $i \geq j$. Suppose further that for every $i\in \N$ there exists a collection $\mathcal{B}_i$ of closed balls for which  satisfies
\[
\mathcal{E}_i \succ_{\xi,\nu} \mathcal{B}_i 
\]
and $\cB_i \subset \cB_j$ for $i\geq j$.

If for some open set $U \subset X$, and every ball $B\subset U$ and any $i\in \N$ it holds that
\[
\mathcal{H}^\nu_{\rad(B)}(\bigcup \mathcal{B}_i|^B_{\rad(B)}) \geq \Delta \nu(B),
\]
and $\lim_{i\to\infty} \rad(\mathcal{B}_i)=0$ then for any ball $B \subset U$ it holds that
\[
\mathcal{H}^\xi_{\rad(B)}(\limsup_{i\to \infty} \mathcal{E}_i \cap B) \geq \beta \xi(B).
\]

\end{theorem}

We highlight the fact that this theorem does not impose any doubling or Alhfors regularity assumption on the base space $X$. Instead, the only assumptions are placed on the gauge functions $\xi$ and $\nu$. However, in verifying that the assumption of the implication, i.e. that the $\nu$-content bound holds, one often needs to resort to some information on the space and its measure.

The Mass Transference Principe of Beresnevich and Velani follows from choosing $\mathcal{E}_i=\{B_k : k \geq i\}$ and $\mathcal{B}_i =\{ B_k^f : k \geq i\}$. Note also that $\limsup_{k\to \infty} B_k^f \cap B \subset \bigcup_{i\geq N} \mathcal{B}_i \cap B$ for each $N\in \N$.  A detailed proof will be presented in Section \ref{subsec:appl-mtp}.

\subsection{Random limsup sets}

We will study the problem of determining the Hausdorff measure and dimension of random limit superior sets in the case of a  metric group $G$. Suppose that $G$   is equipped with a left-invariant metric $d$ (with which it is complete) and that it is unimodular, that is, there exists a bi-invariant Haar measure $\mu$. 
Assume further, that $\mu$ is $D$-doubling for some $D\geq 1$. Let $U$ be any open set in $G$, with $\mu(U) \in (0, \infty)$ and $(E_i)_{i\in \N}$ a sequence of bounded open sets in $G$ with $\lim_{i\to\infty} \diam(E_i)=0$ and $\lim_{i\to\infty}d(\unit_G, E_i)=0$, where $\unit_G$ is the unity element in the group $G$.

We consider the probability measure $P=\frac{\mu|_U}{\mu(U)}$. This generates a probability measure on $U^\N$ given by independently sampling $\omega_i \in U$ for each $i\in \N$ according to the distribution of $P$. Let $\omega = (\omega_i)_{i\in \N}$. 

\begin{theorem}\label{thm:random} If $f:(0,\infty) \to (0,\infty)$ is a dimension function so that the function $r\mapsto \frac{f(r)}{\mu(B(\unit_G,r))}$ is monotone decreasing and $\lim_{r\to 0} \frac{f(r)}{\mu(B(\unit_G,r))}=\infty$, then the following two assumptions are equivalent.

\begin{enumerate}
\item  There exists a constant $\beta>0$ so that for a.e. $\omega=(\omega_i)_{i\in \N}$ we have for any ball $B \subset U$
\[
\mathcal{H}^f_{\rad(B)}(\limsup_{i\to\infty} \omega_i E_i\cap B)\geq \beta f(\rad(B)).
\]
\item 
\[
\sum_{i\in \N} \mathcal{H}^f_\infty(E_i) = \infty.
\]
\end{enumerate} 
\end{theorem}
A difference between Theorem \ref{thm:MTP} and the previous theorem is the assumption on monotonicity. Indeed, Theorem \ref{thm:MTP} is easy to show when $r\mapsto f(r)\mu(B(\unit_G,r))^{-1}$ is monotone increasing, or where the function is monotone decreasing with a finite limit as $r\to 0$. Indeed, there the interesting case occurs when the function is monotone decreasing. We focus on this case to give a cleaner proof and to avoid some technicalities. Further, as formulated, our theorem would not hold in the case of the function $r\mapsto f(r)\mu(B(\unit_G,r))^{-1}$ being monotone increasing.

Consider next the case, when $G$ is $k$-Ahlfors regular. That is, there are constants $C\geq 1$ so that $C^{-1} r^{k} \leq \mu(B(x,r)) \leq C r^k$ for each $x\in G$ and $r>0$. Let $t_0 = \sup\left\{t \in [0,k]: \sum_{i\in \N}\mathcal{H}^t_\infty(E_i)=\infty\right\}$. Then, almost surely, $\mathcal{H}_{\dim}(\limsup_{i\to\infty} \omega_i E_i) \geq t_0$. On the other hand, if $s>0$ and $\sum_{i\in \N}\mathcal{H}^s_\infty(E_i)<\infty$, then a standard Borel-Cantelli argument shows that $\cH^s(\limsup_{i\to\infty} \omega_i E_i)=0$. The latter statement is deterministic and holds for all $\omega_i$. These together prove  the following.

\begin{corollary}
Suppose that $\mu$ is Ahlfors regular, then for almost every $\omega \in U^\N$, it holds that
\[
\mathcal{H}_{\dim}(\limsup_{i\to \infty}\omega_i E_i) = \inf\left\{s>0 : \sum_{i\in \N}\mathcal{H}^s_\infty(E_i)<\infty \right\} = \sup\left\{t>0 : \sum_{i\in \N}\mathcal{H}^t_\infty(E_i)=\infty\right\}.
\]
\end{corollary}

This theorem implies the main theorems in \cite{EJJVrect}. Further, a straightforward modification of our method would also yield versions of the the results of  \cite{FJJV} for open sets. Their result applied also to limsup sets generated by sets $E_i$ with positive Lebesgue density\footnote{Indeed, the proof of Theorem \ref{thm:random} first proves an estimate for content for the union of the sets $E_i$. This portion of the proof is independent of the fact that the sets $E_i$ are open. In the final step openness is used to intersect the sets while preserving the content bounds.}. Our techniques do have implications for their setting, but we do not touch upon them in this paper. Our theorem contains some finer information compared to earlier theorems. Indeed, it can be used for intermediate gauge functions, such as $f(r)=r\log(r)$, and in non-Ahlfors regular settings.

Finally, consider the case $X=\R^k$ and $f(r)=r^s$, with $s\in (0,k)$. Then, Theorem \ref{thm:random} implies that if $\sum_{i\in \N} \mathcal{H}^f_\infty(E_i)=\infty$, then the set $\limsup_{i\to\infty} \omega_i E_i$ is almost surely a set of large intersection in the sense of Falconer. This follows directly from \cite[Theorem B]{falconersets}, where the definition can also be found. A similar conclusion is likely to hold for general groups (and for some metric spaces). However, it seems to the author that a full generalization of Falconer's work in \cite{falconersets} is still missing. Some work, in specific settings, has appeared in \cite{Asets} and in an unpublished manuscript \cite{negreira2021sets}.

\subsection{Sets of finite perimeter}


Sets of finite perimeter can be defined, in Euclidean spaces, as those sets $E \subset \R^d$ for which
\[
\sup \left\{\left| \int {\rm div} \psi(x) 1_E(x) dx\right | : \psi\in C^\infty(\R^d, \R^d), |\psi(x)|\leq 1, \forall x\in \R^d \right\}<\infty,
\]
where $C^\infty(\R^d, \R^d)$ is the collection of all smooth vector fields, ${\rm div}$ is the divergence and $v \to |v|$ is the Euclidean norm on $\R^d$. 

A crucial fact about sets of finite perimeter is that they have many equivalent characterizations. The definition above is equivalent to the following one, which involves relaxations: There exists a sequence $f_i\in C^\infty(\R^d)$ with $\lim_{i\to \infty} \int_{B(0,R)} |f_i(x)-1_E(x)|dx=0$ for every $R>0$ and
\begin{equation}\label{eq:euclidean-relax}
\liminf_{i\to\infty}\int_{\R^d} |\nabla f_i|(x) dx<\infty.
\end{equation}
Here, $\nabla f$ is the gradient of a function $f$.

While these analytic definitions are useful, to develop geometric measure theory, one also needs geometric characterizations. Federer showed in \cite[Theorem 4.5.11]{F}, that $E \subset \R^d$ is a set of finite perimeter if and only if $\cH^{d-1}(\partial^*E)<\infty,$ where $\partial^*E$ is the measure theoretic boundary (see below). Remarkably, Panu Lahti \cite{LFederer} observed that Federer's geometric characterization extends to general metric measure spaces which satisfy doubling and a Poincar\'e inequality, which we will shortly define. We present a \emph{new} proof of this fundamental fact, and our proof is substantially more direct. Indeed, most of the proof is contained in this paper and relies on \cite{Adoubling} only in one of the (previously known) directions. Before stating the characterization precisely, let us review some of the theory on sets of finite perimeter.

The Euclidean theory is classical, and well presented in \cite{evansgariepy} (see also \cite{F} for a much more comprehensive treatment). The remarkable aspect about the definition using relaxations is that it allows for an extension to general metric spaces. Ambrosio developed the theory of finite perimeter sets in metric measure spaces; see \cite{A,Adoubling} for original references and \cite{Asurvey} for a nice survey.
The study of finite perimeter sets is a corner stone of geometric measure theory in general metric spaces. Further, it has provoked much study and had important applications; see e.g. \cite{ChKHeisenberg, FSSC} for imporant works in the context of so called Carnot groups. In metric measure spaces, it has sparked a rich study of functions of bounded variation and isoperimetric inequalities; see e.g. \cite{korte, NL, KKST}. We will give a mostly self-contained proof of Federer's characterization, and introduce only sufficienlty much terminology to state and prove the result.

In metric spaces, one needs a replacement for the gradient appearing in Equation \eqref{eq:euclidean-relax}. This is furnished by the notion of an \emph{upper gradient} of Heinonen and Koskela \cite{heinonenkoskela}.
 A Borel function $g:X \to [0,\infty]$ is said to be an upper gradient for a function $f:X \to \R$, if for every non-constant rectifiable curve $\gamma:[0,1]\to X$ it holds that
\begin{equation}\label{eq:upgrad}
 |f(\gamma(0))-f(\gamma(1))|\leq \int_\gamma g ds,
\end{equation}
 where on the right, we have the curve integral with respect to the length measure on $\gamma$.\footnote{In the case, where the left hand side would evaluate to $|\infty-\infty|$ or $|-\infty-(-\infty)|$, we require $\int_\gamma gds = \infty$.}

With the notion of upper gradient, Ambrosio \cite{Adoubling} defined sets of finite perimeter set. First, let $L^1_{\loc}(X)$ be the space of locally integrable functions. A measurable set $E \subset X$ is a set of finite perimeter, if there exists a sequence $(f_i)_{i\in \N}$ of functions $f_i \in L^1_{\loc}(X)$ with upper gradients $g_i \in L^1$, so that  $\lim_{i\to \infty} \int_{B(0,R)} |f_i(x)-1_E(x)|dx=0$ for every $R>0$ and
\begin{equation}\label{eq:metric-relax}
\liminf_{i\to\infty}\int_{X} g_i(x) d\mu<\infty.
\end{equation}

Without some regularity assumption on the space, sets of finite perimeter can behave quite wildly. For example, if $X$ possesses no non-constant rectifiable curves, the function $g=0$ is an upper gradient for any function, and every  measurable set $E$ is a set of finite perimeter. The proper assumptions, as identified in \cite{Adoubling}, which yield a rich theory are those of a doubling and a Poincar\'e inequality.

Recall, that a metric measure space $(X,d,\mu)$ is said to be $D$-doubling, if 
\begin{equation}\label{eq:doublingdef}
0<\mu(2B) \leq D\mu(B) <\infty\text{ for all (open or closed) balls } B \subset X.
\end{equation}
Further $X$ is said to satisfy a ($1-$)Poincar\'e inequality if there are constants $\lambda, c_{P}>0$ so that for all functions $f \in L^1_{\loc}(X)$, all upper gradients $g$ and all balls $B \subset X$ we have
\begin{equation}\label{eq:PIconstant}
\vint_{B} |f-f_B| \, d\mu \leq c_P \rad(B) \vint_{\lambda B} g \, d\mu,
\end{equation}
where $g_A\defeq \vint_A g d\mu = \frac{1}{\mu(A)} \int_A g d\mu$, whenever $A$ is a measurable set with $\mu(A)\in (0,\infty)$, and $g|_A$ is integrable.
 A space $X$ is called a PI-space, if there are constants $D,C_P,\lambda>0$ so that the previous holds.

The geometric characterization of a set being finite perimeter involves a measure theoretic notion of the boundary. First, define the measure theoretic interior $I_E$ and exterior $O_E$ for a measurable set $E\subset X$ by 
\[
I_E = \left\{ x\in X : \lim_{r\to 0} \frac{\mu(\clB(x,r) \cap E)}{\mu(\clB(x,r))} =1 \right\}, \text{ and } O_E = \left\{ x\in X : \lim_{r\to 0} \frac{\mu(\clB(x,r) \setminus E)}{\mu(\clB(x,r))} =1 \right\}.
\]
The measure theoretic boundary $\partial^* E$ is given by $\partial^* E = X \setminus (I_E \cup O_E)$, which is analogous to how the topological boundary is defined. It is also given by  

 \begin{equation}\label{eq:mboundary}
 \partial^* E = \{x\in X:  \limsup_{r\to 0} \frac{\mu(E\cap \clB(x,r))}{\mu(\clB(x,r))} > 0 ,
 \limsup_{r\to 0} \frac{\mu(\clB(x,r)\setminus E)}{\mu(\clB(x,r))} > 0 \}.
 \end{equation}
 
 The size of this boundary is naturally measured by the co-dimension-one Hausdorff-measure $\cH^h$, which corresponds to the gauge-function $h(\clB)=\frac{\mu(\clB)}{\rad(\clB)}$.
 Federer's characterization in a metric measure space now reads as follows.

\begin{theorem}\label{thm:federer} Let $X$ be a complete PI-space, and let $E \subset X$ be a measurable set. The set $E$ is of finite perimeter if and only if $\cH^h(\partial^* E)<\infty$.
\end{theorem}

\begin{remark}
In fact, the statement could be made slightly more quantitative: There is a constant $\eta$, so that $E$ is a set of finite perimeter if and only if $\cH^h(\partial^*_\eta E)<\infty$. Here, $\partial_\eta^* E$ is a quantitative version of the boundary of $E$:

 \begin{equation}\label{eq:mboundaryeta}
\partial_\eta^* E = \{x \in X: \limsup_{r\to 0} \frac{\mu(E\cap \clB(x,r))}{\mu(\clB(x,r))} > \eta ,
 \limsup_{r\to 0} \frac{\mu(\clB(x,r)\setminus E)}{\mu(\clB(x,r))} > \eta\}.
 \end{equation}
 This version follows easily from our proof and \cite[Theorem 5.4]{Adoubling}.
While we do not go there in this paper, this quantitative aspect has other implications as well. For example, one can follow the proofs here and \cite[Theorem 5.4]{Adoubling} to show that the perimeter measure (which is comparable to $\cH^h$), is supported on the set $\partial^*_\eta E$ -- which may be a strict subset of $\partial^* E$. Further, this would show that $\cH^h(\partial^* E \setminus \partial^*_\eta E)=0$. The $\eta$ here does not depend on the constant $c_P$ in the Poincar\'e inequality, but only on $\lambda$ and the doubling constant $D$. In geodesic spaces, $\lambda$ can be taken to be unity, and then the constant depends only on doubling (see e.g. \cite{hajkos}). This improves the dependence in constants of \cite[Theorem 5.4]{Adoubling}. It indicates, that constants could be sharpened in other results as well.
\end{remark}

We outline the argument here. The bridge between Theorem \ref{thm:inflation-mainthm} and Theorem \ref{thm:federer} is to consider the collection of balls which are half-empty and half-full of the set $E$:
\[
\cB(E,\delta)=\{\clB : \min(\mu(\clB\cap E),\mu(\clB \setminus E)) \geq \delta \mu(\clB)\},
\]
where $\delta>0$ is a fixed constant. A classical argument from \cite{korte} shows that, for all $s>0$,  $\cH^{h}_{\rad(\clB)}(\bigcup \cB(E,\delta)|_{s}^{\clB})$ has a lower bound in terms of $\min(\mu(\clB\cap E),\mu(\clB \setminus E))$; see Lemma \ref{lem:finiteperimunion}. Theorem \ref{thm:inflation-mainthm} then directly applies to give a lower bound for the content of $\limsup \cB(E,\delta)$. 
Finally, a doubling argument given in Lemma \ref{lem:limsupmeastheor} yields that $\limsup \cB(E,\delta) \subset \partial^* E$, which proves the claim.

This argument is quite natural, and avoids much of the machinery used in \cite{LFederer}. However, Lahti has later used his tools in \cite{Lnew} to give a stronger version of Theorem \ref{thm:federer}. We do not know if our techniques would also give a proof of this fact. 

Finally, we note that the proof of Theorem \ref{thm:federer} also directly gives the so called "strong isoperimetric inequality", see Theorem \ref{thm:isoperimetric-ineq}. This was also proved in \cite{LFederer}.

\subsection{Modulus almost every curve}

Underlying the Poincar\'e inequality \eqref{eq:PIconstant} is the geometric property of possessing many curves.
 The richness of curves can be measured in terms of the modulus of a curve family, which we briefly review. A curve is a continuous function $\gamma:[0,1]\to X$. If $\Gamma$ is a collection of non-constant rectifiable curves, then we say that a Borel function $\rho:X \to [0,\infty]$ is admissible for $\Gamma$, and write $\rho \wedge \Gamma$, if $\int_\gamma g \, ds \geq 1$ for each $\gamma \in \Gamma$. We define the $1$-modulus by the expression

\[
\Mod_1(\Gamma) := \inf_{\rho \wedge \Gamma} \int g \,d\mu.
\]

A family of curves $\Gamma$ is said to be $1$-exceptional, if $\Mod_1(\Gamma)=0$. Further, we say that a property $P$ holds for $1$-a.e. rectifiable curve in $\Gamma$, if $\Mod_1(\Gamma \setminus \Gamma_P)=0$, when $\Gamma$ is a family of non-constant rectifiable curves and $\Gamma_P=\{\gamma\in \Gamma: \gamma \text{ satisfies } P\}$. For much more about modulus, as well as the definition of $p-$modulus for other $p\in (1,\infty)$, see \cite{HKST}.

Given the notion of a measure theoretic interior, exterior and boundary, it is geometrically intuitive that curves going form the measure theoretic interior to the exterior \emph{ought} to pass through the measure theoretic boundary. In general, this can fail. However, in a PI-space, this property does not fail too badly, and we can reprove the following statement, which originally appeared  in \cite[Corollary 6.4]{NL}. 

\begin{theorem} \label{thm:generic-curves} Let $X$ be a PI-space.
If $E$ is a set of finite perimeter, then $1$-Mod almost every curve in $\Gamma_{I_E,O_E}$ passes through the measure theoretic boundary.
\end{theorem}

This property is deeply connected to Theorem \ref{thm:federer}, and we refer to \cite[Section 6]{NL} and \cite{LFederer} for a more detailed discussion.

\section{Preliminaries}\label{sec:prelim}

First, we state some of our notational conventions. We write $C=C(a,b,c,\dots)$, when $C$ is a constant which can be bounded by the quantities $a,b,c,\dots$.
When the dependence of a constant is not so important, we write $A \lesssim B$ or $B\gtrsim A$ to indicate that there exists a constant $C$ for which $A\leq CB$, and where $C$ depends only on some universal constants -- such as $D,\lambda$ and $c_P$ from the definition of a PI-space -- and not on specific data, such as a particular set $E$.

Throughout the paper, we will be considering only complete metric spaces $X$. Given an open or closed ball $C$, we denote its radius by $\rad(C)$. For us, a ball $B(x,r)$ is specified by a center $x$ and a radius $r>0$. As sets, two different centers and radii may produce the same set. However, balls will be considered not merely as sets, but with a fixed center and radius. Formally, this would require us to define "balls" as tuples $(x,r)$, and associate to each such tuple a set $B(x,r)$. We abuse notation and identify the set with a tuple, and call both of them balls. Thus, whenever we discuss balls, or collections of balls, each of the balls will have an associated and fixed radius and center.

We remind the reader also that whenever we say ''ball'', we mean either an open or closed ball. We will say open ball, and closed ball, where it is important to restrict to such balls. Further, in order to not introduce odd notation, we use $B$ to denote either any ball, or an open ball, and $\clB$ will exclusively denote closed balls. Finally, where it is important, we will state if a claim holds for both open and closed balls.

Consider a collection $\mathcal{C}$ of open balls.  We define $\rad(\mathcal{C})=\sup_{c\in \mathcal{C}}\rad(C)$.  If $A \subset X$, we denote by $\mathcal{C}_A$ the collection  defined by

\[
\mathcal{C}_A = \{c \in \mathcal{C} : C \cap A \neq \emptyset\}.
\]
Note that $\mathcal{C}_A$ is distinct from $\mathcal{C}|^A$, which was defined in the introduction. We say that $\mathcal{C}$ covers $A$ if $A \subset \bigcup \mathcal{C}$. These notions also can be defined for collections of closed balls.

If $\xi$ is a doubling gauge, we define the (spherical) Hausdorff ($\xi-$)content of a subset $A\subset X$ with the equality

\begin{align}
\mathcal{H}^\xi_r(A) &\defeq \inf \left\{\sum_{c \in \mathcal{C}} \xi(c) : A \subset \bigcup \mathcal{C}, \mathcal{C} \text{ is a collection of open balls with } \rad(\mathcal{C})\leq r\right\}.
\end{align}

An increasing function $f:(0,\infty) \to (0,\infty)$ is called a dimension function. When $\xi(B)=f(\rad(B))$ for some dimension function $f$, then we also use the notation $\mathcal{H}^f_r(A)$ for the content. 

Define the Hausdorff measure by $\mathcal{H}^f(A) = \lim_{\delta \to 0}\mathcal{H}^f_\delta(A)$. When $f(r)=r^s$, for some $s\in (0,\infty)$, we sometimes simplify the notion by writing $\mathcal{H}^s$. We also define the Hausdorff dimension of a set $A$ as

\[
\Hdim(A) := \sup\{t>0 : \cH^t(A)=\infty\}=\inf\{s>0 : \cH^s(A)=0\}.
\]

\subsection{Covering Lemmas}

We recall the following form of the usual 5-covering lemma. See \cite[Section 3.3]{HKST} for a proof. In the statement, the balls may either be closer or open.
\begin{lemma}\label{lem:5covering} Suppose that $X$ is a separable metric space and that $\mathcal{B}$ is a collection of balls with $\sup_{B \in \mathcal{B}}\rad(B) <\infty$. There exists a countable disjoint subcollection $\mathcal{B}' \subset \mathcal{B}$ with 
\[
\bigcup \mathcal{B} \subset \bigcup 5\mathcal{B}',\]
and so that for each $B\in \mathcal{B}$ there is a ball $B'\in \mathcal{B}'$ with $B'\cap B\neq \emptyset$ and $2\rad(B') \geq \rad(B)$.
\end{lemma}

The following lemma is a way to obtain a disjoint covering for a union of sets, while allowing to throw away a small set. 
We use it to reduce the inflation factor in Lemma \ref{lem:5covering}. Before its statement, we comment on the notation $C(\mathcal{B}|_s)$ in the statement. It means that one first restricts to balls with radii less than $s$, and then inflates them. Unfortunately, these operations are not quite commutative. Later, we will see even operations of the form $C(\mathcal{B}|_s^A)$, where we first apply the restrictions to balls $B \subset A$ with $\rad(B)\leq s$ (which do commute with each other), and then inflate. To avoid confusion, we will mostly perform the operations in this order (restrict first and then inflate). There will be only a few exceptions to this rule, and in all instances we will use parenthesis to explicate the order of operations. 

Note that in the statement, the balls may be open or closed.
\begin{lemma}\label{lem:nice-collection}
Suppose that $\mathcal{B}$ is a collection of balls contained in $B_0$, and that $\xi$ is a doubling gauge function with the property that for every disjoint collection $\mathcal{B}' \subset \mathcal{B}$ we have
\[
\sum_{B\in \mathcal{B}'} \xi(B) < \infty.
\]
For every $\epsilon,\Delta \in (0,1)$, $C\geq 1$ there exists a finite disjoint collection $\mathcal{B}_\epsilon \subset \mathcal{B}$ and a positive $s \in (0,\infty)$ so that

\[
\mathcal{H}^\xi_{\rad(\cB)}\left( \bigcup C\left(\mathcal{B}|_s\right) \setminus \bigcup (1+\Delta)\mathcal{B}_\epsilon \right) < \epsilon.
\]
\end{lemma}

\begin{proof}
Let $\mathcal{B}' \subset \mathcal{B}$ be a collection coming from Lemma \ref{lem:5covering}. We have
 \[
 \sum_{B\in \mathcal{B}'} \xi(B) < \infty.
 \]
 
 Thus, by the doubling property of $\xi$, we can choose a finite collection $\mathcal{B}_\epsilon \subset \cB'$ so that
 
 \begin{align} \label{est:choiceofballs}
 \sum_{B\in \mathcal{B}' \setminus \mathcal{B}_\epsilon} \xi(10CB) < \epsilon,
 \end{align}
and so that $\rad(\cB_\epsilon \setminus \cB') \leq \rad(\cB)/(10C)$.
 
 Choose  $s=\min_{B\in \mathcal{B}_\epsilon} \frac{\Delta}{2C} r(B)$. For each $x\in \bigcup_{B\in\mathcal{B}, \rad(B) \leq s} CB $ there exists a ball $B\in \mathcal{B}$ with $\rad(B) \leq s$ with $x\in CB$. By the choice of the collection $\mathcal{B}'$, there is a ball $B' \in \mathcal{B}'$ so that $B' \cap B \neq \emptyset$ and so that $r(B') \geq \frac{1}{2}r(B)$. 
 
 First, suppose that $B'\in \mathcal{B}_\epsilon$. By the definition of $s$, we have $\rad(B) \leq s \leq \frac{\Delta}{2C} \rad(B')$. Then, $CB \subset (1+\Delta)B'$.

Otherwise, suppose that $B' \not\in \mathcal{B}_\epsilon$. Then $B' \in \mathcal{B}' \setminus \mathcal{B}_\epsilon$ and $CB \subset 5CB'$. Thus
 
 \[
 \bigcup_{B\in\mathcal{B}, \rad(B) \leq s} C B  \subset \bigcup_{B\in\mathcal{B}_\epsilon} (1+\Delta) B \cup \bigcup_{B\in\mathcal{B}'\setminus \mathcal{B}_\epsilon} 5CB.
 \]

Then 
\[
\bigcup_{B\in\mathcal{B}, \rad(B) \leq s} CB \setminus \bigcup_{B\in\mathcal{B}_\epsilon} (1+\Delta) B \subset \bigcup_{B\in\mathcal{B}'\setminus \mathcal{B}_\epsilon} 5CB
\]
 and the claim follows from Estimate \eqref{est:choiceofballs}. (The additional factor of $2$ arises, since a closed ball may be needed to be replaced by an open ball of twice the radius.)
\end{proof}






\subsection{Content bounds for unions}

The assumption for our main Theorem involves a content bound for a union of sets. In this section, we give a few crucial lemmas that give bounds for such unions. They all involve an idea that we may transfer a content bound for one collection of sets to a content bound for another collection. In the process, we transition from a gauge $\nu$ to a gauge $\xi$. While this section is a bit technical, we will highlight certain features of the proofs, which will be useful later. Indeed, we aim to structure the proofs in a way that the similarities become apparent.

We start with a useful lemma, which allows us to transition between different scales. We briefly discribe the idea of the statement and its proof structure. Given a set $E \subset B=B(p,R)$, our estimates will involve lower bounds for $\cH^\xi_R(E)$. However, due to the inflation factors from Lemmas \ref{lem:5covering} and \ref{lem:nice-collection}, we often go through a bound for $\cH^\xi_{R'}(E)$ for some $R'< R$. In general, $\cH^\xi_R(E)\leq \cH^\xi_{R'}(E)$ and there is no opposite inequality. However, the only way that $\cH^\xi_R(E)$ is much smaller than $\cH^\xi_{R'}(E)$, is if the infimum involved in defining $\cH^\xi_R(E)$ is realized by a cover $\cC$ with some balls of radius larger than $R'$. 
We call this the \textbf{big ball case}. Many of our arguments will need to consider such a case, and we use this term for any setting where one of the relevant balls is quite large. In each of these cases, one applies the doubling of the gauge function to inflate the bound to the appropriate scale. The remaining case, which we often call the \textbf{small ball case}, is often argued differently.

\begin{lemma}\label{lem:changescale} Let $\xi$ be a $D$-doubling gauge function, $E \subset B$ where $B=B(p,R)$, and $L\geq 1$. Then

\[
\min\left(\cH^{\xi}_{R/L}(E), D^{-\lfloor \log_2(L)\rfloor-8} \xi(B)\right) \leq \mathcal{H}^\xi_{R}(E).
\]
\end{lemma}
\begin{proof}
Let $\cC$ be any cover of $E$ by open balls with $\rad(\cC)\leq R$.  We will show that $\sum_{c\in \cC} \xi(c) \geq 
\min\left(\cH^{\xi}_{R/L}(E), D^{-\lfloor \log_2(L)\rfloor-8} \xi(B)\right)$. Without loss of generality, assume that $c\cap B\neq \emptyset$ for each $c\in \cC$. 

\vskip.3cm
\noindent \textbf{Big ball case:} First, we assume that $\rad(\cC) > R/L$. Then, there exists a $c_b\in \cC$ with $\rad(c_b) > R/L$. Let $k=\lfloor \log_2(\frac{R}{\rad(c_b)}) \rfloor +  1.$ Then, $8R \geq \rad(2^{k+2}c_b)\geq 4R$ and \[B \subset 2^{k+2} c_b \subset B(p,2\rad(2^{k+2}c_b)),\] and by doubling $\xi(c_b) \geq D^{-k-7}\xi(B)$. Therefore, $\sum_{c\in \cC}\xi(c) \geq \xi(c_b) \geq D^{-k-7}\xi(B)$.

\vskip.3cm
\noindent \textbf{Small ball case:} Next, we assume that  $\rad(\cC) \leq R/L$. Then the claim follows from the definition of $\cH^{\xi}_{R/L}(E)$.
\end{proof}

Next, we give the main lemma which is used to transfer a $\nu$-content bound to a $\xi$-content bound. Recall, that
\begin{equation}\label{def:succ}
E \succ_{\xi,\nu} B \,\, \text{ if } \,\, \cH^\xi_{\rad(B)}(E\cap B) \geq \nu(B).
\end{equation}
Further, $\mathcal{E} \succ_{\xi,\nu} \mathcal{B}$ if for every $B\in \mathcal{B}$ there exists an $E\in \mathcal{E}$ so that $E\succ_{\xi,\nu}B$.  Further, recall, that $\nu \ll \xi$ if for all balls $B\subset B'$ with $\rad(B)\leq \rad(B')$, we have 
\begin{equation}\label{def:llrel}
\frac{\nu(B)}{\xi(B)} \leq \frac{\nu(B')}{\xi(B')}.
\end{equation}

In the statement, any ball which appears can be taken as closed or open.
\begin{lemma}\label{lem:transprincip} Let $\cB$ be a collection of balls, and $\cE$ a collection of sets. Suppose that $\mathcal{E} \succ_{\xi,\nu} \mathcal{B}$ and that $\nu$ and $\xi$ are $D$-doubling gauge functions with $\nu \ll \xi$. Then, for a constant $\delta=\delta(D)=D^{-12}2^{-1}>0$ the following property holds. 

For any ball $B$, we have

\[
\mathcal{H}^{\xi}_{\rad(B)}(\bigcup \mathcal{E} \cap B) \geq \delta \min\left(1,\frac{\xi(B)}{\nu(B)}\right) \mathcal{H}^\nu_{\rad(B)}
\left(\bigcup \mathcal{B}|^B_{\rad(B)}\right).
\]
\end{lemma}

\begin{proof} Let $B$ be an arbitrary ball. By Lemma \ref{lem:changescale},
\[
\mathcal{H}^{\xi}_{\rad(B)}(\bigcup \mathcal{E} \cap B) \geq \min\left(\mathcal{H}^{\xi}_{\rad(B)/20}\left(\bigcup \mathcal{E} \cap B\right), D^{-12}\xi(B)\right).
\]
 Since
\[\xi(B) = \frac{\xi(B)}{\nu(B)} \nu(B) \geq \frac{\xi(B)}{\nu(B)} \cH^\nu_{\rad(B)}\left(\bigcup \cB|^B_{\rad(B)}\right),\] it suffices to give a lower bound for $\mathcal{H}^{\xi}_{\rad(B)/20}(\bigcup \mathcal{E} \cap B)$ of the desired form.
 
The proof will be divided to cases which depend on the radii of the balls in $\cB$.

\vskip.3cm
\noindent \textbf{Big ball case:} Assume that $\rad(\cB|^B_{\rad(B)}) > \rad(B)/15$.  By this assumption, there exists a ''big ball'' $B' \in \cB|^B_{\rad(B)}$ with $\rad(B) \geq \rad(B') > \rad(B)/15$. Since $\mathcal{E} \succ_{\xi,\nu}\cB$, there exists a set $E \in \cE$ with $\cH^{\xi}_{\rad(B')}(E \cap B') \geq \nu(B')$. Since $E\in \cE$, we have $E\cap B' \subset \bigcup \cE \cap B$. Further, $ B \subset 2^k B'\subset 4B$ which holds for an integer $k$ with $1\leq k \leq 5$) By doubling of $\nu$ . The definition \eqref{def:succ}, we get 
\[
\cH^{\xi}_{\rad(B)/20}\left(\bigcup \cE \cap B\right) \geq \cH^\xi_{\rad(B')}(E\cap B') \geq \nu(B') \geq D^{-6}\nu(B) \geq D^{-6}\cH^\nu_{\rad(B')}\left(\bigcup \cB|^B_{\rad(B)}\right).
\]

 \vskip.3cm
 \noindent \textbf{Small ball case:} With the proof established in the previous case, we are left to consider the case when $\rad(\cB|^B_{\rad(B)}) \leq \rad(B)/15$.

Suppose that $\mathcal{C}$ is a collection of balls with $\rad(\mathcal{C}) \leq \rad(B)/20$ so that $\bigcup \mathcal{E} \cap B \subset \bigcup \mathcal{C}$. 
We will show that 
\[
\sum_{c\in \cC} \xi(c) \geq \delta  \min\left(1,\frac{\xi(B)}{\nu(B)}\right) \mathcal{H}^{\nu}_{\rad(B)}\left(\bigcup \mathcal{B}|^B_{\rad(B)}\right).
\] Without loss of generality, we can assume that $c \cap B \neq \emptyset$ for all $c\in \cC$. Taking an infimum over all collections $\cC$ then yields the desired bound.

First, we will need to modify the collection $\cB$ in a way to obtain disjoint balls with some separation property. By applying Lemma \ref{lem:5covering} to the collection $3(\mathcal{B}|^B_{\rad(B)})$, we can find a collection $\mathcal{B}' \subset \mathcal{B}|^B_{\rad(B)}$ so that $3\mathcal{B}'$ is a disjoint collection and so that $\bigcup \mathcal{B}|^B_{\rad(B)} \subset \bigcup 15\mathcal{B}'$.  
Since $\bigcup \mathcal{B}|^B_{\rad(B)} \subset \bigcup 15\mathcal{B}'$, we have

\[
\mathcal{H}^\nu_{\rad(B)}\left(\bigcup 15\mathcal{B}'\right) \geq \mathcal{H}^\nu_{\rad(B)}\left(\bigcup \mathcal{B}|^B_{\rad(B)}\right).
\]

The collection $\cC$ covers $\cE \cap B$. Since $\cE \succ_{\xi,\nu} \cB$, then Equation \eqref{def:succ} implies that for every $B'\in \cB'$, there must be a $c\in \cC$ which intersects it. By considering how big the radii of such  intersecting $c$ are, we will then divide our estimates two two cases: the \textbf{inflation case}, and the \textbf{disjoint case}. In the first case, the radius of $c$ is quite large, and the entire ball $B'$ can be easily covered by \emph{inflating} $c$. In the second case, the radius of all the $c$ which which intersect $B'$ are small, and certain sub-collections of $\cC$ are \emph{disjoint}. In both cases, the property guarantees a desired covering property or estimate. See Figure \ref{fig:coveringlemma} for a picture of the argument.

\begin{figure}[h!] 
  \centering
     \includegraphics[width=0.7\textwidth]{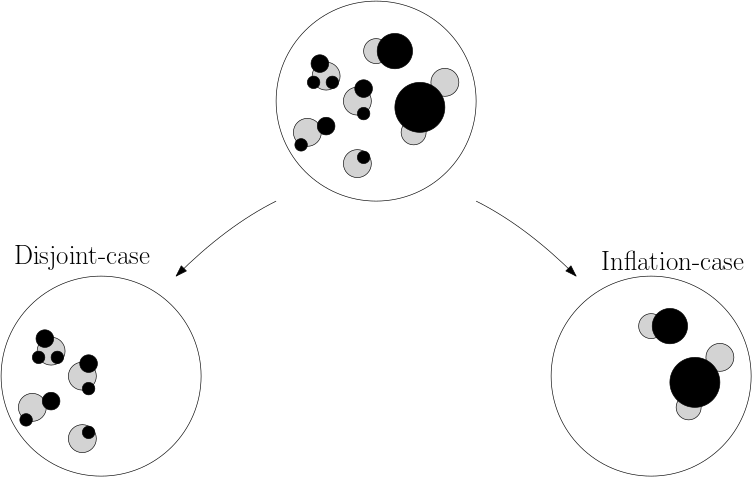}
     \caption{A depiction of the two cases, the \textbf{inflation case} and the \textbf{disjoint case}. On the top is depicted are two sets of balls. The gray ones signify balls in $\cB'$ and the black ones balls in $\cC$. On the left, those balls from $\cB'$ have been chosen, which only intersect balls $c\in \cC$ with small radii. In this case, its crucial to observe that no ball $c\in \cC$ can intersect two of them. This observation uses separation. On the right, the balls $B' \in \cB'$ are shown which intersect some large balls in $\cC$. In this case there is no disjointness, but an inflation can be used to obtain a cover.}
     \label{fig:coveringlemma}
 \end{figure}

The cases will correspond to two sub-collections of $\mathcal{B}'$:
\begin{align*}
\text{\textbf{Inflation: }} \mathcal{B}_b' &= \{B_b \in \mathcal{B}' : \exists c \in \mathcal{C} \text{ with } c \cap B_b \neq \emptyset \text{ and } \rad(c) \geq \rad(B_b)\}, \\
\text{\textbf{Disjoint: }} \mathcal{B}_s' &= \{B_s \in \mathcal{B}' : \forall c \in \mathcal{C} \text{ with } c \cap B_s \neq \emptyset \text{ we have } \rad(c) < \rad(B_s)\}.
\end{align*}

\vskip.3cm 
\noindent \textbf{Inflation case:} Consider first the collection $\cB_b'$. By the definition, and the inequality between radii, we have $\bigcup 15\mathcal{B}_b' \subset \bigcup 20 \mathcal{C}$. In particular, by using doubling and the relation $\nu \ll \xi$ (recall \eqref{def:llrel}), we get

\begin{align}
\sum_{c\in \mathcal{C}}  \xi(c) &= 
\sum_{c\in \mathcal{C}} \nu(c) \frac{\xi(c)}{\nu(c)} \nonumber \\
&\geq \frac{\xi(2B)}{\nu(2B)}  \sum_{c\in \mathcal{C}} \nu(c)\nonumber  \\
&\geq \frac{D^{-7}\xi(B)}{\nu(B)}  \sum_{c\in \mathcal{C}} \nu(20c) \nonumber \\
&\geq \frac{D^{-7}\xi(B)}{\nu(B)} \mathcal{H}^\nu_{\rad(B)}\left(\bigcup 15\mathcal{B}_b'\right).  \label{eq:Bbballs}
\end{align}

\vskip.3cm
\noindent \textbf{Disjoint case}: Consider next the collection $ \cB_s'$. 
For each $B_s \in \mathcal{B}_s'$ define $\mathcal{C}_{B_s} = \{c \in \mathcal{C} : c \cap B_s \neq \emptyset\}$. 
Since for each $c\in \mathcal{C}_{B_s}$ we have $\rad(c) \leq \rad(B_s)$ and $c\cap B_s \neq \emptyset$, we get $c \subset 3B_s.$ 
Let $B_s,B'_s\in \mathcal{B}_s' \in \cB_s'$ be two distinct balls. By the construction of the collection above, $3B_s\cap 3B_s' = \emptyset$. Therefore $\mathcal{C}_{B_s} \cap \mathcal{C}_{B'_s} = \emptyset$. Further, for each $B_s\in \mathcal{B}_s'$, by assumption, we can choose a set $E \in \mathcal{E}$ so that $E \succ_{\xi,\nu} B_s$. By choosing for each $B_s\in \cB_s'$ one such $E$, we define $E_{B_s}=E\cap {B_s}$ for such a ball. 

We have $E_{B_s} \subset \bigcup \mathcal{C}_{B_s}$. Since $E \succ_{\xi,\nu} B_s$, the defining equation \eqref{def:succ} gives
\[
\sum_{C \in \mathcal{C}_B} \xi(C) \geq \nu(B).
\]

Thus, from the pairwise disjointness of $\cC_{B_s}$ for distinct $B_s\in \cB_s'$, we get

\begin{align}
\sum_{C\in \mathcal{C}} \xi(C) &\geq \sum_{B_s\in \mathcal{B}_s'} 
\sum_{C\in \mathcal{C}_{B_s}} \xi(C) \nonumber \\
&\geq D^{-4} \sum_{B_s\in \mathcal{B}_s'}  \nu(15B_s)  \nonumber \\
&\geq D^{-4} \mathcal{H}^\nu_{\rad(B)}(\bigcup 15\mathcal{B}_g'). \label{eq:Bsballs}  
\end{align}

At the end, we combine the estimates from the disjoint and inflation cases. Indeed, by adding \eqref{eq:Bbballs} and \eqref{eq:Bsballs} we get

\begin{align*}
2\sum_{C\in \mathcal{C}} \xi(C) &\geq D^{-7}\min\left(1, \frac{\xi(B)}{\nu(B)}\right) \left(\mathcal{H}^\nu_{\rad(B)}\left(\bigcup 15\mathcal{B}_g'\right)+ \mathcal{H}^\nu_{\rad(B)}\left(\bigcup 15\mathcal{B}_b'\right)\right)\\
&\geq D^{-7}\min\left(1, \frac{\xi(B)}{\nu(B)}\right) \mathcal{H}^\nu_{\rad(B)}\left(\bigcup 15\mathcal{B}'\right) \geq D^{-7}\min\left(1, \frac{\xi(B)}{\nu(B)}\right) \mathcal{H}^\nu_{\rad(B)}(\bigcup \mathcal{B}|^B_{\rad(B)}).
\end{align*}

\end{proof}

A convenient corollary of this is the following result.

\begin{corollary}\label{cor:samegauge} Suppose that $X$ is a complete metric space. Let $\xi$ be a $D$-doubling gauge function and $\eta,\beta \in (0,1)$. Suppose that $\mathcal{B}$ is a collection of balls for which 
\[
\mathcal{H}^\xi_{\rad(B)}(\eta B) \geq \beta \xi(B)
\]
holds for every $B\in \mathcal{B}$. 

Then, for any ball $B_0$ we have
\[
\mathcal{H}^{\xi}_{\rad(B_0)}\left(\bigcup \eta \left(\mathcal{B}|^{B_0}\right)\right) \geq D^{-11}2^{-1} \beta \mathcal{H}^\xi_{\rad(B_0)}\left(\bigcup \mathcal{B}|^{B_0}_{\rad(B_0)}\right).
\]
\end{corollary}

\begin{proof} Let $\nu = \beta \xi$, $\mathcal{E}=\{\eta B : B \in \mathcal{B}|^{B_0}\}$. We have for each $B\in \mathcal{B}$, by assumption, that $\eta B \succ_{\xi,\nu} B$. Thus, $\mathcal{E} \succ_{\xi,\nu} \mathcal{B}$. 
Lemma \ref{lem:transprincip} applied to the collection $\cE$ and the ball $B_0$ gives:
\[
\mathcal{H}^{\xi}_{\rad(B)}\left(\bigcup \eta \left(\mathcal{B}|^{B_0}\right)\right) \geq D^{-11}2^{-1} \min\left(1,\frac{\xi(B_0)}{\nu(B_0)}\right)\mathcal{H}^{\nu}_{\rad(B)}\left(\bigcup \mathcal{B}|^{B_0}_{\rad(B_0)}\right).
\]
Since $\xi(B)/\nu(B)=\frac{1}{\beta}>0$ for all balls, we obtain the claim.
\end{proof}

We will also need the following corollary.

\begin{corollary} \label{cor:set-union} Let $\xi,\nu$ be $D$-doubling gauge functions with $\nu\ll \xi$. With $\delta=2^{-1}D^{-7}$ the following holds. Suppose that $B_0=B(x_0,r_0)$ is a fixed ball, $O\subset B_0$ is an open set,  $A\subset O$ is its subset.
If for any $x\in O$ there is a $r_x \in (0,d(x,X \setminus O))$ so that when $r\in (0,r_x)$ and $B=B(x,r) \subset O$, we have $A \succ_{\xi,\nu} B$ then we have
\[
\mathcal{H}^\xi_{r_0}(A) \geq \delta \frac{\xi(B_0)}{\nu(B_0)}\mathcal{H}^\nu_{r_0}(O).
\]
\end{corollary}
\begin{proof}

Suppose that $\mathcal{C}$ is a collection of balls with $\rad(\mathcal{C}) \leq r_0$ which covers $A$.  We will show that $\sum_{c\in \mathcal{C}} \xi(c) \geq 2^{-1}D^{-7} \frac{\xi(B_0)}{\nu(B_0)}\mathcal{H}^\nu_{r_0}(O).$ Without loss of generality, assume that $c\cap B_0 \neq \emptyset$ for each $c\in \mathcal{C}$. Further, we can assume that $\sum_{c\in \mathcal{C}} \xi(c)<\infty$.  Again, our proof will involve big ball and small ball cases.

\vskip.3cm
\noindent \textbf{Big ball case:} If there exists a $c\in \mathcal{C}$ with $\rad(c) \geq r_0/2$, then $B_0 \subset 8c\subset B(x_0,2\rad(8c))$. In particular $\xi(c) \geq D^{-5} \xi(B_0)\geq D^{-5}\frac{\xi(B_0)}{\nu(B_0)} \nu(B_0) $. Since $\{B_0\}$ is a cover of $O$, we have  $\mathcal{H}^\nu_{r_0}(O) \leq \nu(B_0)$. Combining these observations yields the claim. 

\vskip.3cm
\noindent \textbf{Small ball case:} Thus, in what follows, we may assume that $\rad(\mathcal{C}) \leq r_0/2$.

Let $\bfB = O \setminus \bigcup 2\mathcal{C}$. We will divide the rest of the proof to two cases. First, consider the case $\mathcal{H}^\nu_{r_0}(\bfB)\leq \frac{1}{2}\mathcal{H}^\nu_{r_0}(O).$ The opposite inequality will be considered after this. By sub-additivity of the Hausdorff-content, we have $\mathcal{H}^\nu_{r_0}(\bigcup 2\mathcal{C} \cap O) \geq \frac{1}{2} \mathcal{H}^{\nu}_{r_0}(O).$

Since $2\mathcal{C}$ is a cover of $\bigcup 2\mathcal{C} \cap O$ and $\rad(2\cC)\leq 2\rad(\cC)\leq r_0$ (by our setup), we have
\[
\sum_{c\in \cC} \nu(2c) \geq \frac{1}{2} \mathcal{H}^{\nu}_{r_0}(O).
\]
Since $\xi\ll \nu$, $\xi$ and $\nu$ are $D$-doubling and since $c\subset 2B_0$, for each $c\in \mathcal{C}$, 
\begin{align*}
\sum_{c \in \cC} \xi(2c) &= \sum_{c\in \cC} \frac{\xi(2c)}{\nu(2c)} \nu(2c) \\
&\geq D^{-2}\sum_{c\in \cC} \frac{\xi(c)}{\nu(c)} \nu(2c) \\
&\geq D^{-2} \frac{\xi(B_0)}{\nu(B_0)} \sum_{c\in \cC}  \nu(2c) \\
&\geq \frac{D^{-2}}{2} \frac{\xi(B_0)}{\nu(B_0)} \mathcal{H}^{\nu}_{r_0}(O).
\end{align*}
This yields the desired estimate in the case $\mathcal{H}^\nu_{r_0}(\bfB)\leq \frac{1}{2}\mathcal{H}^\nu_{r_0}(O).$ 

Next, consider the case $\mathcal{H}^\nu_{r_0}(\bfB)\geq \frac{1}{2} \mathcal{H}^\nu_{r_0}(O).$ This case will lead to a contradiction, which proves the claim. Choose, $\sigma=D^{-11}2^{-3} \min(1,\frac{\xi(B_0)}{\nu(B_0)} )\mathcal{H}^\nu_{r_0}(O)$. 

Since $\sum_{c\in \cC}\xi(c)<\infty$, we can find a finite set $\cC'$ so that $\cC' \subset \cC$ with
\begin{equation}\label{eq:cpchoice}
\sum_{c\in \cC' \setminus \cC} \xi(c) \leq \sigma.
\end{equation}

Let $O'=O \setminus \bigcup_{c\in \cC'} \overline{c}$ and $O'' = O \setminus \bigcup 2\cC'$.  We have $\mathbf{B} \subset O'' \subset O'$, and thus 
\[
\mathcal{H}^\nu_{r_0}(O') \geq \mathcal{H}^\nu_{r_0}(\mathbf{B}) \geq \frac{1}{2} \mathcal{H}^\nu_{r_0}(O).\] 
Let now $\mathcal{E}=\{A\cap O'\}$, and take $\mathcal{B}=\{B : B=B(x,r_x) \subset O' \text{ for some } x\in O'\}$. By the assumption on $r_x$, $\mathcal{E} \succ_{\xi,\nu} \mathcal{B}$. Also, $O' = \bigcup \mathcal{B}|^{B_0}_{r_0}$, since $O'$ is open. Therefore, by Lemma \ref{lem:transprincip} applied to the collections $\cE,\cB$ and the ball $B_0$, we get

\[
\cH^{\xi}_{r_0}(A \cap O') \geq D^{-11}2^{-1} \min\left(1,\frac{\xi(B_0)}{\nu(B_0)}\right) \cH^{\nu}_{r_0}(O').
\]

Since $\C$ is a covering of $A$, $\cC \setminus \cC'$ is a covering of $A\cap O'$, and therefore

\[
\sum_{c\in \cC\setminus \cC'} \xi(c) \geq D^{-11}2^{-1} \min\left(1,\frac{\xi(B_0)}{\nu(B_0)}\right) \cH^{\nu}_{r_0}(O') \geq D^{-11}2^{-2} \min\left(1,\frac{\xi(B_0)}{\nu(B_0)}\right) \cH^{\nu}_{r_0}(O).
\]
This is a contradiction to the choice of $\cC'$, $\sigma$ and \eqref{eq:cpchoice}.
\end{proof}

\subsection{Gauges coming from measures}

In this subsection we will focus on the case, where our transference gauge $\nu$ is given by a Borel measure on $X$ which is finite and positive on all balls. That is, there is a Borel measure $\overline{\nu}$ so that $\nu(B)=\overline{\nu}(B)$ for all open balls and $\overline{\nu}(B)\in (0,\infty)$ for every ball $B \subset X$. We do not assume that boundaries of balls have zero measure, and we define the gauge for closed balls $\clB=\clB(x,r)$ by setting $\nu(\clB)=\nu(B)$, where $B=B(x,r)$. Since $\nu$, and therefore $\overline{\nu}$, will be assumed to be $D$-doubling, this convention only affect the values up to a constant factor. In particular, for all balls $B$, whether open or closed, it holds that $\nu(B)\leq \overline{\nu}(B)\leq D \nu(B)$.

By a slight abuse of notation, we use $\nu$ also to denote this measure $\overline{\nu}$, and say that the gauge $\nu$ comes from a measure. 

Given a set $E \subset X$, let $\mathcal{B}_{E,\xi,\nu} = \{B(x,r) : E \succ_{\xi,\nu} B(x,r) , x\in E\}$, and define a set 
\begin{equation}\label{eq:transf-set-def}
B_E = \bigcup \mathcal{B}_{E,\xi,\nu}.
\end{equation}
With a few assumptions on $\nu$ and $\xi$, we can give precise bounds on the volume of $B_E$. These play a central role in proving our theorem on random limsup sets, Theorem \ref{thm:random}.

In this Lemma, the balls used will all be open.  First, a brief remark about the statement.

\begin{remark}\label{rmk:content-gauges}
If $\xi(B)=f(\rad(B))$ for a dimension function $f$, then we could replace $\cH^\xi_R(E)$ by $\cH^\xi_\infty(E)$ in the statement. This follows, because the assumption $E\subset B(p,R)$ implies that $\cH^\xi_\infty(E)=\cH^\xi_R(E)$. Indeed, $\cH^\xi_\infty(E)\leq \cH^\xi_R(E)$ for all gauge functions. However, $\cH^\xi_\infty(E)\leq f(R)$ since $E \subset B(p,R)$. Thus, the infimum in $\cH^\xi_\infty(E)$ is obtained by a covering $\cC$ with $\rad(\cC)\leq R$, and thus $\cH^\xi_\infty(E)\geq \cH^\xi_R(E)$. In fact, for such gauges, we could remove the factor $\frac{\xi(B(p,R))}{\nu(B(p,R))}$ from the upper bound. The only change to the proof comes in the proof of the upper bound. We will indicate this modification with a remark.
\end{remark}

\begin{lemma}\label{lem:transfer_set_measure} Assume that $\nu,\xi$ are $D$-doubling gauge functions. Suppose further that
\begin{itemize}
\item $\nu \ll \xi$, 
\item $\nu$ comes from a measure, 
\item $\lim_{r\to 0} \nu(B(x_0,r))=0$, and
\item $\lim_{r \to 0} \frac{\xi(B(x_0,r))}{\nu(B(x_0,r))} = \infty$ for each $x_0 \in X$.
\end{itemize}
 There exists a constant $C$ so that for every bounded open set $E$, $p \in X$ and $R>0$, with $E\subset B(p,R) \subset X$, we have $E \subset B_E$ and either $\cH^\xi_R(E) \geq \nu(X)$ or
\[
\frac{1}{C}\nu(B_E) \min\left(1,\frac{\xi(B(p,R))}{\nu(B(p,R)))}\right)\leq \cH^{\xi}_{R}(E) \leq C \frac{\xi(B(p,R))}{\nu(B(p,R))} \nu(B_E).
\]
\end{lemma}

\begin{proof}

 For each $x\in E$ define $r_x=\sup\{r \in (0,\infty) : \cH^\xi_{r}(B(x,r)\cap E) \geq \nu(B(x,r))\}$. We will first argue that $r_x \in (0,\infty]$ for each $x\in E$ by showing that the set defining $r_x$ is non-empty. 

\vskip.3cm
\noindent \textbf{Non-emptyness:} Fix $x\in E$. By openness of $E$ and the assumptions of the statement, there exists a radius $r_1 \in (0,\infty)$ so that $\xi(B(x,r_1))\geq \nu(B(x,r_1))$ and so that $B(x,r_1) \subset E$. Choose $r_2 \in (0,r_1/2)$ so that $\nu(B(x,r_2)) \leq \xi(B(x,r_1))D^{-8}$. We argue that $r_x\geq r_2$ by showing that $\cH^\xi_{r_2}(B(x,r_2) \cap E) \geq \nu(B(p,r_2))$. 

Let $\mathcal{C}$ be any cover of $B(x,r_2)=B(x,r_2) \cap E$ with $\rad(\mathcal{C}) \leq r_2$. We will show that $\sum_{c \in \mathcal{C}} \xi(c)\geq \nu(B(p,r_2))$, which by the definition of $r_x$ gives $r_x \geq r_2$. Without loss of generality, we assume that $c\cap B(x,r_2) \neq \emptyset$ for each $c\in \mathcal{C}$. There are two cases to consider:
\begin{enumerate}
\item \textbf{Big ball case,} $\rad(C) > r_1/4$: There exists a $c_b\in C$ with $\rad(c_b) \geq r_1/4$. By doubling, since $B(x,r_1) \subset 2^4 c_b\subset B(x,2^5 \rad(c_b))$, we get $\xi(c_b) \geq D^{-8}\xi(B(x,r_1))$.  This gives $\sum_{c \in \mathcal{C}} \xi(c)\geq D^{-8}\xi(B(x,r_1)) \geq \nu(B(p,r_2))$ by the choice of $r_2$.
\item \textbf{Small ball case,} $\rad(C) \leq r_1/4$: In this case $c \subset B(x,r_1)$ for each $c\in \mathcal{C}$, and thus \[
\xi(c)/\nu(c) \geq \xi(B(x,r_1))/\nu(B(x,r_1))\geq 1\]
 since $\nu \ll \xi$. Thus, $\mathcal{C}$ is a cover of $B(x,r_2)$, and since $\nu$ comes from a measure, we get
\[\sum_{c \in \mathcal{C}} \xi(c) \geq \sum_{c \in \mathcal{C}} \nu(c) \geq \nu(B(x,r_2)).\]
\end{enumerate} 

Since $r_x>0$ for each $x\in E$, we obtain $E \subset B_E$. Further, $\cH^{\xi}_{r_x}(B(x,r_x) \cap E) \geq \nu(B(x,r_x))$ follows by continuity of measure and since the function $\phi: r \mapsto \cH^\xi_{r}(B(x,r)\cap E)$ satisfies the inequality $\limsup_{s\to r^{-}} \phi(s)\leq \phi(r)$ for all $r \in (0,\infty]$. It is then straightforward to show that 
 \begin{equation}\label{eq:Beequality}
B_E = \bigcup_{x\in E} B(x,r_x).
 \end{equation} 
 
We have either $\sup_{x\in E} r_x =\infty$ or $\sup_{x\in E} r_x<\infty$.

\vskip.3cm
\noindent \textbf{Infinite radius case:} If $\sup_{x\in E} r_x =\infty$, there is a sequence of $(x_i)_{i\in \N}$ with $x_i \in E$ for which $\lim_{i\to \infty}r_{x_i}=\infty$. Then $\nu(X) = \lim_{i\to \infty} \nu(B(x_i,r_{x_i})) \leq \cH^\xi_{R}(E)$. The lemma has been proven, since then $\nu(X) \leq \cH^\xi_R(E)$. 

\vskip.3cm
\noindent \textbf{Finite radius case:} Thus, in what follows we restrict to the case  $\sup_{x\in E} r_x<\infty$ and prove the content bound for $B_E$. This is divided to the lower and upper bounds for the volume $\nu(B_E)$.

\hrulefill

\vspace{.5em}
\noindent \textbf{Lower bound:} There are two sub-cases: either  $\sup_{x\in E} r_x \in (R,\infty)$ or $\sup_{x\in E} r_x \leq R$. 

In the first case, choose a $\textbf{x} \in E$ with $r_\textbf{x} > \max(\sup_{x\in E} r_x/2,R)$. Then, $\cH^\xi_{r_{\bfx}}(B(\bfx,r_{\bfx})\cap E) \geq \nu(B(\bfx, r_{\bfx}))$. By equation \eqref{eq:Beequality}, we  have also for any $y\in B_E$, that $d(y,\bfx) < 2r_{\bfx}+2R \leq 6r_{\bfx}$. Thus, $B_E \subset B(\bfx, 6r_{\bfx})$, and the lower bound follows:
\[
\cH^\xi_R(E) \geq \cH^\xi_R(B(\bfx,r_{\bfx})\cap E)  \geq \cH^\xi_{r_{\bfx}}(B(\bfx,r_{\bfx})\cap E)\geq \nu(B(\bfx, r_{\bfx})) \geq D^{-3} \nu(B(\bfx, 8r_{\bfx})) \geq D^{-3}\nu(B_E).
\]

Next, consider the case $\sup_{x\in E} r_x \leq R$. Let $B=B(p,2R)$, $\mathcal{E}=\{E\}$ and let \[\mathcal{B}=\{B_x\defeq  B(x,r_x) : x \in E\}.\]  We have $B_E \subset B(p,2R)$ and $B_x\defeq B(x,r_x) \subset B(p,2R)$ for each $x \in E$. Then, with this and since $\nu$ comes from a measure, we get 
\begin{equation}\label{eq:B_Econtent}
\cH^\nu_{2R}\left(\bigcup \mathcal{B}|^B_{\rad(B)}\right)=\nu(B_E).
\end{equation}

Lemma \ref{lem:transprincip} with the previous equation gives a $\delta>0$ for which
\[
\cH^\xi_{2R}(E)=\cH^\xi_{2R}(E \cap B) \geq \delta \min\left(1,\frac{\xi(B)}{\nu(B)}\right) \nu(B_E).
\]

Since $B_E \subset B(p,2R)$, we have $\nu(B_E) \leq D\nu(B(p,R))$, and the desired lower bound for  $\cH^\xi_R(E)$ then follows from Lemma \ref{lem:changescale} applied with $L=2$.

\hrulefill

\vspace{.5em}
\noindent \textbf{Upper bound:} Again, we need to consider the cases $\sup_{x\in E} r_x \leq R/5$ and $\sup_{x\in E} r_x > R/5$. In the second case, there exists an $x\in E$ with $r_x>R/5$. Then,

\[
\nu(B_E) \geq \nu(B(x,r_x)) \geq D^{-4} \nu(B(x,2^4 r_x)) \geq D^{-4} \nu(B(p,R)).
\]
We have \[
\cH^\xi_R(E) \leq \xi(B(p,R))\leq \frac{\xi(B(p,R))}{\nu(B(p,R))} \nu(B(p,R)),\] since $B(p,R)$ covers $E$. From this, the upper bound follows.

Next, assume $\sup_{x\in E} r_x \leq R/5$. We have $E \subset B_E \subset \bigcup \mathcal{B}$. Apply Lemma \ref{lem:5covering} to the collection $\mathcal{B}$, to obtain a collection of balls $\mathcal{B}'\in \mathcal{B}$ with $\bigcup \mathcal{B}\subset \bigcup 5\mathcal{B}'$ and so that for any two distinct balls $B_1,B_2 \in \mathcal{B}'$ we have $B_1\cap B_2 = \emptyset$. 

First, by the definitions of $\mathcal{B}$ and of $r_x$, and since $r_x\leq R/5$, we have for any $B' \in \cB'$ that $\cH^\xi_{R}(E) \leq \cH^\xi_{\rad(5B')}(E \cap 5B') < \nu(5B')$. By these facts, and subadditivity,
\begin{align*}
\cH^\xi_R(E) &\leq \sum_{B' \in \cB'} \cH^\xi_R(5B' \cap E) \\
&\leq \sum_{B' \in \cB'} \nu(5B') \\
&\leq D^3\sum_{B' \in \cB'} \nu(B') \leq D^3\nu(B_E).
\end{align*}

\begin{remark}
If $\xi(B)=f(\rad(B))$ for some dimension function $f$, we could replace $R$ with $\infty$ in the proof of the upper bound, and avoid considering the case of $\sup_{x\in E} r_x>R/5$. In that case, the proof would give a simpler upper bound  of $C\nu(B_E)$.
\end{remark}
\end{proof}

\section{Proofs of main results}

In this section we give the proofs of the main Hausdorff content theorems. We show the content bounds in a somewhat unconventional way. We first consider a collection of balls $\mathcal{C}$ so that $\sum_{c\in C} \xi(c)$ is quite small. Then, we find a point $p \in \limsup (1+\delta)\mathcal{B}$ that is not covered. The contrapositive of this statement is that any covering of $\limsup (1+\delta)\mathcal{B}$ must have the corresponding lower bound for $\sum_{c\in C} \xi(c)$. By definition, this yields the content bound for the limsup set.

 The point $p$ is found in one of two ways. In the first case, we find $p$ within an intersection of a sequence of poorly covered balls, which we call ``bad balls'' -- This is called the \textbf{nested bad balls case}. In the second case, we find $p$ within an intersection of compact sets $K_k$. The crucial insight is that if the first case does not apply, that is, there does not exist such a nested sequence of bad balls, then the conditions of Lemma \ref{lem:nice-collection} are satisfied, and we obtain $K_k$ as a union of finitely many poorly covered balls. The second case is called the \textbf{nested sets case}. See Figure \ref{fig:hlemmacases} for a depiction of these cases.

\begin{figure}[h!] 
  \centering
     \includegraphics[width=0.5\textwidth]{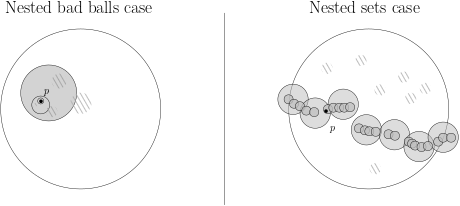}
     \caption{A depiction of the two cases, the \textbf{nested bad balls case} and the \textbf{nested sets case}. The dashed regions indicate the balls in $\cC$, and the gray balls indicate balls in $\cB$. In the nested bad balls case, depicted on the left, one can find a nested sequence of balls poorly covered by $\cC$. In the nested sets case, such a sequence can not be found. However, this forces a certain sum to be finite, which permits us to construct levels of balls, at different scales. In either case, one finds a point $p$ which can not be covered by $\cC$.}
     \label{fig:hlemmacases}
 \end{figure}

For purposes of later sections in this paper, we present a slightly more technical formulation of our main theorem. 

\begin{theorem}\label{thm:inflation-technicalversion}Let $\xi$ be a $(D-)$doubling gauge function and $\Delta,\delta \in (0,1]$ any constants.  There exists a constant $\beta=\beta(\delta, D)$ so that the following holds. 
Suppose that $X$ is a complete metric space and that $\mathcal{B}_i$, $i\in \N$, are collections of closed balls with $\lim_{i \to \infty} \rad(\mathcal{B}_i) = 0$ and $\mathcal{B}_j \subset \mathcal{B}_i$ for $j\geq i$. If
 for each $i\in \N$, $\clB_0=\clB(x_0,r_0) \in \mathcal{B}_i$, $j \geq i$ we have

\begin{equation}\label{eq:content-assumption}
\mathcal{H}^\xi_{r_0}\left(\bigcup \, \mathcal{B}_{j}|^{\clB_0}\right) \geq \Delta  \xi(\clB_0),
\end{equation}
then for every $B_0\in \bigcup_{i\in \N} \mathcal{B}_i$ we have

\[
\mathcal{H}^\xi_{r_0}( \limsup_{i\to \infty} (1+\delta)(\mathcal{B}_{i}|^{\clB_0}))\geq \Delta \beta \xi(B_0).
\]
\end{theorem}

Before proving this theorem, we show how the slightly simpler version in the introduction follows from it.

\begin{proof}[Proof of Theorem \ref{thm:inflation-mainthm}]
The result follows from Theorem \ref{thm:inflation-mainthm-2} by setting $\mathcal{B}_i = \mathcal{B}|_{2^{-i}}$.
\end{proof}

\begin{proof}[Proof of Theorem \ref{thm:inflation-technicalversion}]

Choose  any $L$ so that $2^L \geq \frac{2^5}{\delta}$. Fix $\beta=D^{-L}2^{-2}$.

It suffices to prove the following claim:

\begin{center} Suppose that $i\in \N$, $\clB_0 \in \mathcal{B}_i$ and that $\mathcal{C}$ is a collection of open balls with
\begin{equation}\label{eq:Cbound}
\sum_{c\in\mathcal{C}} \xi(c) < \beta \Delta \xi(B_0).
\end{equation}
Then, there exists a point $p \in \limsup_{j\to\infty} (1+\delta)(\mathcal{B}_j|^{B_0}) \setminus \bigcup\mathcal{C}$.
\end{center} 

In other words, such a collections $\mathcal{C}$ can not cover $\limsup_{j\to\infty} (1+\delta)(\mathcal{B}_j|^{B_0})$ and we get the lower bound for the Hausdorff content via contradiction.

Fix such a collection $\mathcal{C}$ which satisfies estimate \eqref{eq:Cbound} and a ball $\clB_0 \in \mathcal{B}_i$. Recall, that for any ball $\clB$ we define $\mathcal{C}_\clB$ by $\mathcal{C}_{\clB}=\{c \in \mathcal{C} : c \cap \clB \neq \emptyset\}$. The proof will be divided into two cases, which are depicted in Figure \ref{fig:hlemmacases}.

The first case is called the \textbf{nested bad balls case}. In it, we will find the point $p$ within an intersection of a nested sequence of certain ``bad balls''. In the other case, which we call the \textbf{nested sets case}, we will assume that not enough bad balls exist. This will force a certain sum to be finite, which in turn yields a nested sequence of compact sets whose intersection will contain our point $p$.

A ball $\clB \in \mathcal{B}_j$, for $j> i$, is called $(\beta,\delta)$-\emph{bad}, if $\rad(\mathcal{C}_{\clB} )\leq \frac{\delta}{8}\rad(\clB)$ and 
\[
\sum_{c\in\mathcal{C}_\clB} \xi(c) < \beta \Delta \xi(\clB).
\]
 Being bad means that the ball does not intersect any $c\in \mathcal{C}$ which is too large, and that the covering $\mathcal{C}_\clB$ is just as bad as for the original ball $\clB_0$. Set $i_0=i$ in order to simplify notation below.

\vskip.2cm
\noindent \textbf{Nested bad balls case}: 
\vskip.2cm

In the first case, we assume that for each $j\geq i$ and each $\clB\in \mathcal{B}_j|^{\clB_0}$ which is $(\beta,\delta)$-bad, there exists a index $k$ with $k>j$ and  a ball
$\clB'\in \mathcal{B}_k|^{\clB}$ which is 
also $(\beta,\delta)$-bad.
 If this is the case, then since $\clB_0$ is $(\beta,\delta)$-bad,
  we can find an index $i_1$ with $i_1>i_0$ and a ball $\clB_1\in \mathcal{B}_{i_1}|^{\clB_0}$ which is $(\beta,\delta)$-bad. Repeating the argument, 
  we obtain a sequence of $(\beta,\delta)$-bad balls $(\clB_k)_{k\in \N}$ and a strictly increasing sequence $(i_k)_{k\in\N}$ with the 
  properties $\clB_{k+1}\in \mathcal{B}_{i_{k+1}}|^{\clB_k}$ for all $k\in \N$. Note, that $\rad(\clB_{k}) \leq \rad(\mathcal{B}_{i_k})$, and thus, $\lim_{k\to \infty} \rad(\clB_k)=0$. By completeness of $X$, the intersection $\cap_{k\in\N}\clB_k$ is non-empty and consists of only one point. Let $p$ be the unique point in the intersection $\cap_{k\in\N}\clB_k$. We have $p \in  \limsup_{i\to\infty} (1+\delta)(\mathcal{B}_i|^{B_0})$. 
  
  We argue by contradiction that $p\not\in \bigcup \mathcal{C}$. Indeed, suppose that $p\in \bigcup\mathcal{C}$. Since $\mathcal{C}$ is a collection of open sets,   there is some $k$ so that we have $\clB_k \subset C$ for some $C\in \mathcal{C}$ and $\rad(\clB_k) < \frac{\delta}{8} \rad(C)$, which would be a contradiction to being $(\beta,\delta)$-bad. Thus, $p\not\in \bigcup \mathcal{C}$ as claimed, and the proof of the first case is completed.
   
  \vskip.2cm
  \noindent \textbf{Nested sets case:}
  \vskip.2cm
  
If the nested bad balls-case does not apply, this means that there must exist some index $j$ with $j\geq i_0$ and some ball $\clB_1\in \mathcal{B}_j|^{\clB_0}$ which is $(\beta,\delta)$-bad, and for which the following holds.

\begin{center} 
For any index $k>j$ there does not exist a $(\beta,\delta)$-bad ball $\clB$ with $\clB \in \mathcal{B}_k|^{\clB_1}$. 
\end{center}
In order to cite this property later,  we refer to this property as the \emph{non-existence of bad balls} property.

Define $K_1 = (1+\delta)\clB_1$. Let also $r_1 = s_1 = \rad(\clB_1)$, $j_1=j_1'=j$ and $\mathcal{A}_1 = \{\clB_1\}$. The argument will construct a  sequence of compact sets $(K_l)_{l\in\N}$ consisting of non-empty compact sets with a content lower bound. The ability to construct them, will rest on a certain sum from Lemma \ref{lem:nice-collection} being finite.

For each $k>j$, define two collections of balls by the formulas $\mathcal{B}_k^{b}\defeq\{\clB\in \mathcal{B}_k: \rad(\mathcal{C}_\clB) < \frac{\delta}{8} \rad(\clB)\}$ and  $\mathcal{B}_k^{g} \defeq \mathcal{B}_k\setminus \mathcal{B}^b$. We have, for any ball $\clB \in \mathcal{B}_i$:
\begin{align}\label{eq:all-balls}
\bigcup \mathcal{B}_k|^{\clB_1} \subset \bigcup \mathcal{B}_k^b|^{\clB_1} \cup \bigcup\mathcal{B}_k^g|^{\clB_1},
\end{align}
and 
\begin{align}\label{eq:goodballs}
\bigcup \mathcal{B}^g_k|^{\clB_1} \subset \bigcup \frac{2^5}{\delta}\mathcal{C}_{\clB_1}.
\end{align}
The latter follows since for every $B\in \mathcal{B}^g_k|^{\clB_1}$ there exists a ball $c \in \mathcal{C}^{\clB}$ with $\rad(B) \leq \frac{8}{\delta} \rad(c)$ and $B \cap c \neq \emptyset$.

We have by the choice of $L$ and $\beta$ that
\[
\sum_{C\in \mathcal{C}_{\clB_1}} \xi\left(\frac{2^5}{\delta}C  \right) \leq D^L \sum_{C\in \mathcal{C}_{\clB_1}} \xi(C) \leq D^L \beta \Delta \xi(\clB_1).
\]

 Consider the set-relations \eqref{eq:all-balls} and \eqref{eq:goodballs}, and the preceding inequality in conjunction with the  assumed inequality \eqref{eq:content-assumption}. Together with the sub-additivity of Hausdorff-content, these give for all $k > j$
 
 \begin{equation*}
 \mathcal{H}^\xi_{\rad(\clB_1)}\left(\bigcup \mathcal{B}^b_k|^{\clB_1}\right) \geq  \mathcal{H}^\xi_{\rad(\clB_1)}\left(\bigcup \mathcal{B}_k|^{\clB_1}\right)-\mathcal{H}^\xi_{\rad(\clB_1)}\left(\bigcup \mathcal{B}_k^g|^{\clB_1}\right) \geq (1-D^L \beta) \Delta \xi(\clB_1).
 \end{equation*}
 
 Consequently, for any $r \in (0,r_1)$, we have
 \begin{equation}
 \label{eq:lowerboundbadunion}
 \mathcal{H}^\xi_{\rad(\clB_1)}\left(\bigcup (1+\delta) \left(\mathcal{B}_k^b|^{\clB_1}\right)\right) \geq (1-D^L \beta) \Delta\xi(\clB_1).
 \end{equation}

We now verify the assumptions of Lemma \ref{lem:nice-collection} for the collection $(1+\delta/4)(\mathcal{B}^b_k|^{\clB_1})$ for each $k>j$. Suppose that $\mathcal{B}'\subset (1+\delta/4)(\mathcal{B}^{b}_k|^{\clB_1})$ is a collection of disjoint balls. Each ball $\clB\in \mathcal{B}'$ satisfies $\frac{1}{1+\delta/4} \clB\in \mathcal{B}_k|^{\clB_1}$. By the non-existence of bad balls property, the ball $\frac{1}{1+\delta/4}\clB$ is not $(\beta,\delta)$-bad. Also, $\rad(C_{\frac{1}{1+\delta/4}\clB}) \leq \frac{\delta}{8}\rad(\frac{1}{1+\delta/4}\clB)$ since $\frac{1}{1+\delta/4}\clB \in \mathcal{B}^b$.
 Consequently, by the definition of badness, for each $\clB\in \mathcal{B}'$ we have 
 \[
 \sum_{C\in\mathcal{C}_{\frac{1}{1+\delta/4}\clB}} \xi(C) \geq \beta \Delta \xi\left(\frac{1}{1+\delta/4}\clB\right) \geq \frac{\beta \Delta}{D}\xi(\clB).
 \]
 Also, since each $C\in \mathcal{C}_{\frac{1}{1+\delta/4}\clB}$ satisfies $\rad(C) \leq \frac{\delta}{8} \rad(\frac{1}{1+\delta/4}\clB)=\frac{\delta}{8+2\delta}\rad(\clB)$, we have $C \subset \clB$. Indeed, the collections $\mathcal{C}_{\frac{1}{1+\delta/4}\clB}$ for different $\clB \in \mathcal{B}'$ are disjoint. Thus,

\[
\infty > \sum_{C\in \mathcal{C}_{\clB_1}}\xi(C) \geq \sum_{\clB\in \mathcal{B'}}\sum_{C\in \mathcal{C}_{\frac{1}{1+\delta/4}\clB}}\xi(C) \geq \sum_{\clB \in \mathcal{B'}} \frac{\beta \Delta}{D} \xi(\clB).
\]

The previous estimate shows that the assumptions of Lemma \ref{lem:nice-collection} are satisfied for the collection $(1+\delta/4)(\mathcal{B}_k^b|^{\clB_1})$ and all of its sub-collections.

We will recursively construct
\begin{itemize}
\item scale parameters $r_{l},s_l$,
\item indices $j_l,j_l'$,
\item compact sets $K_l,$
\item and collections $\mathcal{A}_l$, 
\end{itemize}
for $l\in \N$, with the following four properties.
\begin{enumerate}
\item[P1] \[
\mathcal{A}_l \subset (1+\delta)\left(\mathcal{B}^b_{j_l}|^{\clB_1}\right) \text{ is finite.}\]
\item[P2] \[K_l = \bigcup \mathcal{A}_l.\]
\item[P3] \[s_{l+1}\leq r_{l+1}\leq s_l/2, \hskip1cm j_{l+1} > j_l' > j_l, \hskip1cm  \rad(\mathcal{B}_{j_{l}})\leq r_{l}.\]
\item[P4] \[\mathcal{H}^\xi_{\rad(\clB_1)}\left(\bigcup (1+\delta)\left(\mathcal{B}_{j_l'}^b|^{\clB_1}\right) \setminus K_l \right) \leq \Delta 8^{-l}\xi(\clB_1).\]
\end{enumerate}

Here, P1 and P2 describe the fact that the compact sets $K_l$ arise as finite collections of balls; see Figure \ref{fig:hlemmacases}. Property P3 ensures that the scales in the construction decrease and that the indices keep pace with these scales. Finally, P4 means that we do not throw away too much in the construction of $K_l$.

\vskip.3cm
\noindent \textbf{Recursive construction:} Proceed as follows. We have already defined the objects for $l=1$. Suppose that $r_{l},s_l,j_l K_l,\mathcal{B}_l$, have been defined. Let $C=\frac{1+\delta}{1+\delta/4}$.

\begin{itemize}
\item Choices of scale parameters and $j_{l+1}$: Define $r_{l+1}=s_{l}/2$ and choose $j_{l+1}>j_l'$ so large that $\rad(\mathcal{B}_{j_{l+1}})\leq r_{l+1}$.  
Above, we established that $(1+\delta/4)\left(\mathcal{B}_{j_{l+1}}^b|^{\clB_1}\right) \subset (1+\delta/4)\left(\mathcal{B}_{k}^b|^{\clB_1}\right)$ satisfies the assumptions of Lemma \ref{lem:nice-collection}.

\item Choice of compact sets and collection: Apply Lemma \ref{lem:nice-collection} with $C=\frac{1+\delta}{1+\delta/4}$, $\Delta = \delta/4$ and $\epsilon = \beta \Delta 2^{-l-1}$ and the collection $\mathcal{B}''=(1+\delta/4)\left(\mathcal{B}^b_{j_{l+1}}|^{\clB_1}\right)$. 
This gives a finite collection $\mathcal{B}_{l+1}' \subset (1+\delta/4)\left(\mathcal{B}_{j_{l+1}}^b|^{\clB_1}\right)$ and a scale $s_{l+1}\leq r_{l+1}$ so that
\begin{equation}\label{eq:whatisleft}
\mathcal{H}^\xi_{\rad(\clB_1)}\left(\bigcup C \left(\mathcal{B}''|_{s_{l+1}}\right) \setminus  \bigcup (1+\delta/4)\mathcal{B}_{l+1}'\right) \leq \Delta 8^{-l}\xi(\clB_1).
\end{equation} 
Now, set $ \mathcal{A}_{l+1}=\frac{1+\delta}{(1+\delta/4)}\mathcal{B}_{l+1}'$, and note that $\frac{1+\delta}{1+\delta/4}>(1+\delta/4)$, and so 
\begin{equation}\label{eq:inflationballs}
\bigcup (1+\delta/4)\mathcal{B}_{l+1}' \subset \bigcup \mathcal{A}_{l+1}. 
\end{equation}
Set $K_{l+1} = \bigcup \mathcal{A}_{l+1}$. 

\item Choice of $j_{l+1}'$: Choose $j_{l+1}'$ so that $j'_{l+1} < j_l$ and so that $\rad(\mathcal{B}_{j'_{l+1}})<s_{l+1}(1+\delta/4)^{-1}$.
\end{itemize}
Properties P1--3 are immediate from these choices. 

Next, we prove property P4.  From the definition and the nested property of the sequence $(\mathcal{B}_n)_{n\in \N}$ we have
\[(1+\delta)(\mathcal{B}^b_{j_{l+1}'}|^{\clB_1}) \subset C\left( \left.\left((1+\delta/4)\left(\mathcal{B}^b_{j_{l+1}}|^{\clB_1} \right)\right)\right|_{s_{l+1}} \right)=C(\mathcal{B}''|_{s_{l+1}}).\]
By \eqref{eq:inflationballs}, we have  $\bigcup(1+\delta/4) \mathcal{B}_{l+1}' \subset K_{l+1}$.
From these, we get
\[
\bigcup (1+\delta)\left(\mathcal{B}_{j_{l+1}'}^b|^{\clB_1}\right) \setminus K_{l+1} \subset \bigcup C(\mathcal{B}''|_{s_{l+1}}) \setminus \bigcup (1+\delta/4) \mathcal{B}_{l+1}'.
\]

Thus, P4 follows from \eqref{eq:whatisleft}.  

\vskip.3cm
\noindent \textbf{Completing the proof:} With the recursion   completed, we show two content bounds for each $l\in \N$. 
From properties P4, P1, P2, subadditivity and estimate \eqref{eq:lowerboundbadunion}, we get for every $l\in \N$ that

\begin{align}
\mathcal{H}^\xi_{\rad(\clB_1)}(K_l) &\geq \mathcal{H}^\xi_{\rad(\clB_1)} \left(\bigcup (1+\delta)\left( \mathcal{B}^b_{j_l'}|^{\clB_1}\right)\right) -  \mathcal{H}^\xi_{\rad(\clB_1)}\left(\bigcup (1+\delta)\left(\mathcal{B}_{j_l'}^b|^{\clB_1}\right) \setminus K_l\right) \nonumber \\
&\geq (1-D^L \beta-8^{-l}) \Delta\xi(\clB_1)\geq \frac{\Delta}{2}  \xi(\clB_1). \label{eq:Kkbound}
\end{align}
For each two integers $k,l \in \N$ with $1\leq l < k$, we get by applying properties P1--3   together with $j_k > j_l'$ that

\begin{equation}
K_k \setminus K_l \subset \bigcup (1+\delta) \left(\mathcal{B}^b_{j_k}|^{\clB_1}\right) \setminus K_l \subset \bigcup (1+\delta)\left( \mathcal{B}^b_{j_l'}|^{\clB_1}\right) \setminus K_l 
\end{equation}
and thus by property P4
\begin{align}\label{eq:smallcut}
\mathcal{H}^\xi_{\rad(\clB_1)}(K_k\setminus K_l) \leq  \Delta 8^{-l}\xi(\clB_1).
\end{align}

Next, let $k\in\N$ be arbitrary and define $\overline{K}_k =\cap_{l=1}^k K_l$. By definition, we have $K_k \subset \overline{K}_k \cup \bigcup_{l=1}^{k-1} K_k \setminus K_l$.  Using inequalities \eqref{eq:smallcut} and \eqref{eq:Kkbound} we get that
\begin{align*}
\mathcal{H}^\xi_{\rad(\clB_1)}(\lK_k) &\geq \mathcal{H}^{\xi}_{\rad(\clB_1)}(K_k)-\sum_{l=1}^{k-1}  \mathcal{H}^{\xi}_{\rad(\clB_1)}(K_k \setminus K_{l}) \\
&\geq \frac{\Delta}{2}  \xi(\clB_1) - \sum_{l=1}^{k-1} \Delta 8^{-l}\xi(\clB_1) \\
&\geq \frac{\Delta}{4}  \xi(\clB_1).
\end{align*}





Now, $(\lK_k)_{k\in\N}$ is a nested sequence of non-empty closed sets. If $X$ was assumed compact, or proper, then each set $\cK_k$ would be compact, and we could conclude that  $\cap_{k\in \N} \lK_k \neq \emptyset$. However, our space $X$ is only complete, and we need to proceed slightly differently. The argument mimics showing that a totally bounded metric space satisfies the finite intersection property.

We say that a ball $\clB \in \mathcal{A}_l$ is good, if $\clB$ intersects $\lK_k$ for infinitely many $k$. Since $\mathcal{A}_1$ is finite and $\cK_k$ is a non-empty set for every $k\in \N$, there must exist a good ball $\cB^1 \in \mathcal{A}_1$. By proceeding recursively, for every $l\in \N$, there must exit a good ball $\cB^l \in \mathcal{A}_l$ so that $\cB^l \cap \cB^{l-1} \neq \emptyset$. Property P3 then implies that the centers $x^i$ of the balls $\cB^i$ form a Cauchy sequence, and converge to some point $p$ by completeness. It is direct to show that $p\in \bigcap_{k\in \N} \lK_k$.

By using $\lK_k \subset K_k$ and property P1, we get that $p \in (1+\delta) (\mathcal{B}_{j_k}^b|^{\clB_1})$ for every $k\in \N$. Consequently, $p \in \limsup_{j\to\infty} (1+\delta) (\mathcal{B}_j|^{\clB_1}) \subset  \limsup_{j\to\infty} (1+\delta) (\mathcal{B}_j|^{\clB})$. The second case of the proof is complete once we show that $p\not\in \bigcup C$.

Indeed, if $p\in \mathcal{C}$, then there is some $c\in\mathcal{C}$ with $p\in c$ and $\rad(c)>0$. Choose $k \in \N$ so that $r_k\leq \rad(c)$. There is some $\clB_p\in (1+\delta)\left(\mathcal{B}^b_{j_k}|^{\clB_1}\right)$ for which $p\in \clB_p$. Also, $c\cap \clB_p \neq \emptyset$ since $p$ lies in both. Finally, since $(1+\delta)^{-1}\clB_p \in \mathcal{B}_{j_k}^b|^{\clB_1}$, we must have $\rad(c) \leq \frac{\delta}{(1+\delta)8} \rad(\clB_p) \leq \frac{r_k}{8}$  which is a contradiction to the choice of $k$.
\end{proof}

\subsection{Refined versions}\label{subsec:refinedvers}

We present refined versions of the main result, where the inflation factor is placed in the assumption.

\begin{theorem}\label{thm:withoutinflation}Let $\xi$ be a doubling gauge function and $\Delta,\delta \in (0,1)$ any constants.  There exists a constant $\beta$ depending on $\delta$ so that the following holds.
Suppose that $X$ is a complete metric space, and that for each $\clB_0=\clB(x_0,r_0) \in (1-\delta)\mathcal{B}$  and any $r\leq r_0$ we have

\[
\mathcal{H}^\xi_{r_0}\left(\bigcup (1-\delta)\left(\mathcal{B}|_{r}^{\clB_0}\right)\right) \geq \Delta  \xi(\clB_0).
\]

Then, for each $\clB_0\in \mathcal{B}$ we have

\[
\mathcal{H}^\xi_{r_0}(\limsup \mathcal{B}|^{\clB_0})\geq \Delta \beta \xi(\clB_0).
\]
\end{theorem}

\begin{proof} Consider the collection $\mathcal{B}'=(1-\delta)\mathcal{B}$. Note that $(1-\delta)(\mathcal{B}|^{\clB_0}) \subset \left((1-\delta) \cB\right)|^{\clB_0}$. Apply Theorem \ref{thm:inflation-mainthm} to this collection. This gives a constant $\beta>0$ so that for all $\clB_0 \in \cB$
\[
\cH^\xi_{(1-\delta)\rad(\clB_0)}\left(\limsup (1+\delta)(\cB|^{(1-\delta)\clB_0})\right) \geq \Delta \beta \xi((1-\delta)\clB_0).
\]

Now $(1-\delta)(1+\delta)<1$, and the claim follows by adjusting $\beta$ from noting that $\limsup (1+\delta)(\cB|^{(1-\delta)\clB_0}) \subset \limsup \cB|^{\clB_0}$ and using the doubling property of $\xi$.
\end{proof}

From this statement, we obtain a variant of our main theorem. 

\begin{proof}[Proof Theorem \ref{thm:inflation-mainthm-2}] 
By the assumption, for any ball $B\in \cB$, and any $r>0$ we have
\[
\mathcal{H}^\xi_{\rad(B)}(2^{-1}B)\geq
\mathcal{H}^\xi_{\rad(B)}\left(\bigcup \, \mathcal{B}|_{r}^{2^{-1}\clB_0}\right) \geq \Delta  \xi(\clB_0).
\]


From the assumption of the statement and Corollary \ref{cor:samegauge}, we obtain for any ball $\clB_0\in 2^{-1}\cB$ and any $r\in (0,\rad(\clB_0))$
\[
\mathcal{H}^\xi_{\rad(\clB_0)}\left(\bigcup \, \frac{1}{2}(\mathcal{B}|_{r}^{\clB_0})\right) \geq \gamma\Delta  \xi(\clB_0)
\]
with a constant $\gamma=\gamma(D,\beta).$ Consequently, the claim then follows from Theorem \ref{thm:withoutinflation} applied to the collection $2^{-1}\cB$.

\end{proof}

\section{Applications: Mass Transference Principle}
\label{sec:applications-1}

\subsection{Mass Transference Principle}
\label{subsec:appl-mtp}
The proof of our general Mass Transference Principle follows now from the technical version of our main theorem, Theorem \ref{thm:inflation-technicalversion}.

\begin{proof}[Proof of Theorem \ref{thm:genmasstransf}]  Assume that $\mathcal{E}_i, \mathcal{B}_i$ are as in the statement and $i\in \N$ arbitrary. Assume also that for every ball $B \subset U$ we have
\begin{equation}\label{eq:nucontent}
\cH^\nu_{\rad(B)}\left(\bigcup \cB_i|^B_{\rad(B)}\right) \geq \Delta \nu(B).
\end{equation}

Let $B\subset U$ be an  arbitrary open ball. First, $\bigcup \mathcal{B}_i|^B_{\rad(B)} \subset B$ for any ball $B \subset U$. Therefore $\mathcal{H}^\nu_{\rad(B)}(B) \geq \Delta \nu(B)$ for any ball $B \subset U$. 

Apply Lemma \ref{lem:transprincip} to obtain a constant $\delta$ so that for $\mathcal{E}=\mathcal{E}_i$ and $\mathcal{B}=\mathcal{B}_i$ we have

\begin{align}
\mathcal{H}^\xi_{\rad(B)}\left(\bigcup \mathcal{E}_i \cap B\right) &\geq \delta  \min\left(1,\frac{\xi(B)}{\nu(B)}\right) \mathcal{H}^\nu_{\rad(B)}\left(\bigcup \mathcal{B}_i|^B_{\rad(B)}\right)  \nonumber \\
&\geq \delta \Delta \min\left(1,\frac{\xi(B)}{\nu(B)}\right) \nu(B). \label{eq:prelimboundE_j}
\end{align}

Recall, that for any open ball $B' \subset B$ we have $\frac{\xi(B')}{\nu(B')} \geq \frac{\xi(B)}{\nu(B)}$. Define a new gauge by $\overline{\nu} = \delta \Delta \min\left(1, \frac{\xi(B)}{\nu(B)}\right) \nu$. Then, $\overline{\nu}(B') \leq \xi(B')$ for all $B' \subset B$ with $\rad(B') \leq \rad(B)$. Further, inequality \eqref{eq:prelimboundE_j} then reads that $\cH^\xi_{\rad(B')}(\bigcup \mathcal{E}_i \cap B') \geq \overline{\nu}(B')$. 
In other words, we get for $A=\bigcup \mathcal{E}_i$ that $A \succ_{\xi,\overline{\nu}} B'$ for each $B' \subset B$ with $\rad(B') \leq \rad(B)$.

By Corollary \ref{cor:set-union} applied to the set $A$ and to $O=B$ we have a constant $\delta'$
\begin{align*}
\mathcal{H}^\xi_{\rad(B)}\left(\bigcup \mathcal{E}_i \cap B\right) &\geq \delta' \frac{\xi(B)}{\overline{\nu}(B)} \mathcal{H}^{\overline{\nu}}_{\rad(B)}\left(\bigcup \mathcal{B}_i|^B_{\rad(B)}\right)
\end{align*}

Since $\overline{\nu} = C' \nu$ for some constant $C'$, we have 
\[\frac{\xi(B)}{\overline{\nu}(B)} \mathcal{H}^{\overline{\nu}}_{\rad(B)}\left(\bigcup \mathcal{B}_i|^B_{\rad(B)}\right)=\frac{\xi(B)}{\nu(B)} \mathcal{H}^{\nu}_{\rad(B)}\left(\bigcup \mathcal{B}_i|^B_{\rad(B)}\right).\]
 Thus, combining this with   \eqref{eq:nucontent} yields
\begin{equation}\label{eq:desiredbound}
\mathcal{H}^\xi_{\rad(B)}\left(\bigcup \mathcal{E}_i \cap B\right) \geq \delta' \Delta \xi(B).
\end{equation}
Note that this inequality holds for any open ball $B\subset U$.

Define $\overline{\cB_0}$ to be the collection of all closed balls contained in $U$.
Define $\mathcal{\overline{B}}_i = \{\clB: \exists E \in \mathcal{E}_{i+1} \text{ s.t. } 2\clB \subset E, \rad(B) \leq 2^{-i-1}\}$. For for any closed or open ball $B$ we have
\begin{equation}\label{eq:setqualityEi}
\limsup_{i\to \infty} 2(\mathcal{\overline{B}}_i|^B)\subset \limsup_{i\to \infty} \mathcal{E}_i \cap 2B. 
\end{equation}
Further, for each $i\in \N$ the collection $\mathcal{\overline{B}}_i$ consists of closed balls and $\mathcal{\overline{B}}_i \subset \mathcal{\overline{B}}_j$ for $i<j$ and $\lim_{i\to\infty} \rad(\mathcal{\overline{B}}_i)=0$. 

Let $\clB=\clB(x,r)\subset U$ be an arbitrary closed ball, and $B=B(x,r)$.
We have $\bigcup \cE_{i+1}\cap B \subset \bigcup \mathcal{\overline{B}}_i|^{\clB}$. Thus, by equation \eqref{eq:desiredbound}, we get
\[
\cH^\xi_{\rad(\clB)}\left(\bigcup \cB_i|^{\clB}\right)\geq \cH^\xi_{\rad(B)}\left(\bigcup \cE_{i+1}\cap B\right) \geq \delta'\Delta \xi(\clB).
\]

Therefore, the assumptions of Theorem \ref{thm:inflation-technicalversion} are satisfied. Therefore, with 
 $\delta=1$ there exists a  $\beta>0$ so that for any closed ball $\clB_0 \in \overline{\cB_0}$ (i.e. for any $\clB_0 \subset U$) we have
  \[\cH^\xi_{\rad(\clB_0)}\left(\limsup_{i\to\infty} 2(\mathcal{\overline{B}}_i|^{\clB_0})\right)\geq \Delta \beta \delta' \xi(\clB_0).\]

Recall that $2\clB \subset \cE_{i+1}$ for each $\clB\in \cB_i$. Thus, 
\[\cH^\xi_{\rad(\clB_0)}\left(\limsup_{i\to \infty} \mathcal{E}_i \cap \clB_0\right)\geq \Delta \beta \delta' \xi(\clB_0),\]
from which the claim follows.
\end{proof}

We indicate how Theorem \ref{thm:genmasstransf} implies Theorem \ref{thm:MTP}.

\begin{proof}[Proof of Theorem \ref{thm:MTP}] 

We may assume that $\rad(B_i)\leq 1$ for all $i\in \N$. Let $\cE_i=\{B_j=B(x_j,r_j) : j \geq i\}$, and $\cB_i = \{B_j^f(x_j,f(r_j)^{1/k})\}$.

First, consider the case if $x\to f(x)x^{-k}$ is increasing. This is the ''easy case'', and we handle it essentially similarly to \cite{beresnevichvelani}. 
Either $\lim_{x\to 0} f(x)x^{-k}=0$ or $\lim_{x\to 0} f(x)x^{-k}=c>0$. In the first case, $\cH^f(B)=0$ for all balls $B$, and the claim is immediate. In the latter case, $f(x)\geq cx^k$, and we get $\cH^f(B)=c\cH^k(B)$. If we apply Corollary \ref{cor:samegauge} to the gauge $\xi(B)=\rad(B)^k$, and use $k$-Alhfors regularity to verify the assumption, we get for some constant $\gamma>0$ and any ball $B_0$ that
\[
\cH^k_{\rad(B_0)}\left(\bigcup_{B_i^f \subset B_0} c^{\frac{1}{k}} B_i^f\right) \geq \gamma \xi(B_0).
\]
However, $c^{\frac{1}{k}} B_i^f \subset B_i$ when $f(x)\geq cx^k$, and we get
\[
\cH^k_{\rad(B_0)}\left(\bigcup \cE_i \right) \geq \gamma \xi(B_0)
\]
for all balls $B_0$. From these, and Theorem \ref{thm:genmasstransf} applied with $\nu=\xi$ and the collections $\cB_i=\cE_i$, the theorem follows with little effort.

Thus, the case when $x\to f(x)x^{-k}$ is monotone decreasing remains. Then set $\xi(B)=f(\rad(B))$ and $\nu(B)=\rad(B)^k$. We have $\xi \ll \nu$, and the claim follows from Theorem \ref{thm:genmasstransf} applied to the collections $\cE_i$ and $\cB_i$, since it is easy to see that $\cE_i \succ_{\xi,\nu} \cB_i$. (Here, the reader may want to recall the defining inequality in \eqref{def:succ}.)\end{proof}

\subsection{Random Limsup sets}

In this subsection, we prove the main result on random limit superior sets, Theorem \ref{thm:random}. Recall, that $1_G$ denotes the identity element of the group $G$. 

\begin{lemma}\label{lem:volume-case} Let $G$ be a unimodular group,  $d$ a left-invariant metric on $G$ and $\mu$ a bi-invariant Haar measure on $G$. Suppose that $U$ is any open bounded set with $\mu(U)<\infty$, and that 
$P=\frac{\mu|_U}{\mu(U)}$. 
Let $E_i \subset G$, for $i\in \N$, be open bounded sets with $\lim_{i\to\infty} \diam(E_i) = \lim_{i\to\infty} d(\unit_G, E_i)=0$. Then, the following two conditions are equivalent.

\begin{enumerate} 
\item For
$P^\N$-a.e. $\omega=(\omega_i)_{i\in \N} \in U^\N$ and any ball $B \subset U$ we have
\[
\mu(\limsup_{i\to \infty} \omega_i E_i \cap B) = \mu(B)
\]
\item \[\sum_{i\in \N} \mu(E_i) = \infty.\]
\end{enumerate}
\end{lemma}
\begin{proof}
Suppose that $\sum_{i\in \N} \mu(E_i) < \infty$. Then a standard Borel-Cantelli argument shows that $\mu(\limsup_{i\to \infty} \omega_i E_i)=0$.

Thus, assume that $\sum_{i\in \N} \mu(E_i) = \infty$. Fix $x\in U$. Right multiplication by $x$ is continuous, and thus we can choose $N_x\in \N$ so that for $i\geq N_x$ we have $\diam(E_ix^{-1})+d(\unit_G,E_ix^{-1}) < d(x,G\setminus U)$. Let $i\geq N_x$ be arbitrary.  
Then, consider the (Borel measurable) events $A_i(x)=\{\omega_i: x\in \omega_i E_i\} = \{\omega_i: \omega_i^{-1} \in E_i x^{-1}\}$. Since $\diam(E_ix^{-1})+d(\unit_G,E_ix^{-1}) < d(x,G\setminus U)$, we have $E_ix^{-1} \subset U$. Thus,

\begin{align*}
P(A_i(x)) &= \frac{\mu(\{\omega_i: \omega_i^{-1} \in E_i x^{-1}\}) }{\mu(U)}\\
&= \frac{\mu(\{\omega_i: \omega_i \in E_i x^{-1}\}) }{\mu(U)}\\
&= \frac{\mu(E_i x^{-1}) }{\mu(U)}=\frac{\mu(E_i)}{\mu(U)}.
\end{align*}
On the penultimate line we used the unimodularity of $G$, which implies that for any measurable subset $A \subset G$ we have $\mu(\{x\in G: x\in A\})=\mu(\{x \in G: x^{-1} \in A\})$. On the final line we used the right-invariance of the Haar measure on unimodular groups.

Since this holds for all $i\geq N_x$, we have $\sum_{i=1}^\infty P(A_i(x))=\infty$. 
Since $A_i(x)$ are independent events, by Borel-Cantelli, for $P^\N$ a.e. $\omega=(\omega_i)_{i\in \N}$ we have $A_i(x)$ for infinitely many $i\in \N$. 
Let $B=\{(x, \omega) \in U \times U^\N : \omega \in A_i(x) \text{ for infinitely many } i\in \N\}$. 
The set $B$ is Borel measurable since the sets $B_i=\{(x,\omega)\in U \times U^\N: \omega_i \in A_i(x)\}$, for $i\in \N$, are all Borel, and $B= \bigcap_{N=1}^\infty \bigcup_{i\geq N} B_i$. 
Since for a.e. $x\in U$, we have that $\{\omega \in U^N : (x,\omega) \in B\}$  has full measure, then $B$ has full $\mu \times P$-measure  in $U \times U^\N$. 
Therefore, by Fubini, for almost every $\omega \in U^\N$, we have that $\{x\in U: (x,\omega) \in B\}=\limsup_{i\to\infty} E_i \cap U$ has full $\mu$-measure.
\end{proof}

This statement, together with Lemma \ref{lem:transfer_set_measure} leads to a proof of Theorem \ref{thm:genmasstransf}. In the proof, we first replace the sets $E_i$ with the sets $B_{E_i}$, whose $\mu$-volumes can be compared to the content of the sets $E_i$. The limit superior set of $B_{E_i}$ has full $\mu$-measure, and Theorem \ref{thm:genmasstransf} allows us to transform this information to the sets $E_i$.

\begin{proof}[Proof Theorem \ref{thm:random}]

We may assume, by possibly passing to the tail of the sequence of sets, that $\diam(E_i) \leq 1$ for all $i\in \N$. Further, let $R>1$ be so that $U \subset B(1_G,R)$.

Define the gauge $\nu$ given by $\nu(B)=\mu(B)$ (for open balls), and $\xi(B)=f(\rad(B))$.  By assumption $\lim_{r\to 0} f(r)=0$ and $\lim_{r\to 0} \mu(B(\unit_G,r))=0$.

If
$\sum_{i=1}^\infty \cH^\xi_\infty(E_i)=0$, then a Borel-Cantelli argument shows that $\cH^\xi_r(\limsup_{i\to\infty} \omega_i E_i) = 0$. This proves that claim (1) implies (2). Next, we prove (2) implies (1), by assuming $\sum_{i=1}^\infty \cH^\xi_\infty(E_i)=\infty$.

The gauge $\nu$ comes from a measure.
Recall the comparison sets defined in Equation \eqref{eq:transf-set-def}: for each $A$ we have a set $B_A$ with $A \subset B_A$. By left-invariance of the metric $B_{\omega A}=\omega B_A$ for each $\omega \in G$.  Further, if $\diam(A) \leq r$, then $\cH^\xi_\infty(A) \leq f(r)$. 
 Thus, $\lim_{i\to\infty} \cH^\xi_\infty(E_i) = 0$. 
 
 As in Remark, \ref{rmk:content-gauges}, we then have $\cH^\xi_\infty(A)=\cH^\xi_R(A)$ whenever $\diam(A)\leq R$.
Apply this to get the following. 
If $x_i\in E_i$ and $R_i>0$ are arbitrary with $E_i \succ_{\xi,\nu} B(x,R_i)$, then 
\[\nu(B(x_i,R_i)) \leq \cH^{\xi}_{R_i}(E_i \cap B(x,R_i)) \leq \cH^\xi_\infty(E_i \cap B(x,R_i)) \leq \cH^\xi_\infty(E_i).\] Since $\nu$ is a Haar measure and the metric is left invariant, all balls have the same volume. 
The function $h: r \mapsto \nu(B(\unit_G,r))$ is increasing, with $\lim_{r\to 0} h(r)=0$. Since $\lim_{i\to\infty} \cH^\xi_\infty(E_i) = 0$, we have $\lim_{i\to\infty} h(R_i) = 0$ and $\lim_{i\to\infty} R_i = 0$.  This, together with the definition of $B_{E_i}$ implies $\lim_{i\to \infty}\diam(B_{E_i}) = 0$. 

By increasing $R$ we can assume that for all $i\in\N$ we have $B_{\omega_i E_i} \subset B(\unit_G,R)$ and so that $\omega_i E_i \subset B(\unit_G,R)$.  Further, by possibly passing to the tail, we can assume that $\cH^{\xi}_R(E_i) < \nu(X)$ for all $i\in \N$.

By Lemma \ref{lem:transfer_set_measure}, we have a value $C\geq 1$ (which depends also on $B(\unit_G,R)$) so that for all $i\in \N$, we have $E_i \subset B_{E_i}$ and
\begin{equation}
\frac{1}{C}\mu(B_{E_i}) \leq \cH^\xi_R(E_i) \leq C\mu(B_{E_i}). 
\end{equation}

In particular, we have $\lim_{i\to\infty}d(B_{E_i},\unit_G)=0$. Since $\cH^\xi_\infty(E_i) \leq \cH^\xi_R(E_i)$, we get
\[
\sum_{i=1}^\infty \mu(B_{E_i})=\infty.
\]
Now, Lemma \ref{lem:volume-case}, implies that for a.e. $\omega\in U^\N$ it holds that $\limsup_{i\to\infty} \omega_i B_{E_i} = \limsup_{i\to\infty} B_{\omega_i E_i}$ has full measure in $U$. Let $\omega\in U^\N$ be such that this occurs.

Let $\cE_i = \{\omega_j E_j : j \geq i\}$ and $\cB_i = \{B_{\omega_j E_j} : j \geq i\}$. Directly, we get $\cE_i \succ_{\xi,\nu} \cB_i$, $\cE_i \subset \cE_j$ for $i\leq j$. Since $\lim_{i\to \infty}\rad(B_{E_i})=0$, we get
\[\cH^\nu_{\rad(B)}\left(\bigcup \cB_i|^B_{\rad(B)}\right) \geq \cH^\nu_{\rad(B)}\left(\limsup_{i\to\infty} \cB_i|^B_{\rad(B)}\right)= \nu(B)\]
for every $B \subset U$. Therefore, by Theorem \ref{thm:genmasstransf}, we get our claim:
\[
\cH^\xi_{\rad(B)}\left(\limsup_{i\to\infty} \omega_i E_i \cap B\right)= \cH^\xi_{\rad(B)}\left(\limsup_{i\to\infty} \cE_i \cap B)\geq \beta f(\rad(B)\right)
\]
for each $B \subset U$.
\end{proof}

\section{Applications: Sets of finite perimeter}

\subsection{Preliminaries}

To be self-contained, we give a few basic properties of sets of finite perimeter. 

\vspace{1em}
\begin{center}\textit{
Throughout this section $(X,d,\mu)$ will be a PI-space. Recall the definitions involved \eqref{eq:doublingdef} and \eqref{eq:PIconstant}. Further, in this section, $\lambda, c_P$ and $D$ will denote the constants from these definitions.}
\end{center}
\vspace{1em}
We will use the gauge defined for closed balls by $h(\clB)=\frac{\mu(\clB)}{\rad(\clB)}$, which yields the co-dimension-one Hausdorff contents $\cH^h_s$, for $s>0$ and the co-dimension-one Hausdorff measure $\cH^h$. For open balls $B=B(x,r)$, we define $h(B)=h(\clB(x,r))$ in order to enforce our condition that
 $h(B)=h(\clB)$ whenever $\rad(B)=\rad(\clB)$.

The geometric characterization of sets of finite perimeter involves considering a quantity measuring the relative density of a subset $E\subset X$, and its complement, at a given scale. For a (closed or open) ball $B$, we denote this quantity with
\[
\Theta(E,B)= \frac{\min(\mu(E\cap B), \mu(B\setminus E))}{\mu(B)}.
\]
Sets $E$ with $\Theta(E,B) = \frac{1}{2}$ for a given ball $B$, can be thought to be "half-full and half-empty" -- although doubling will force us to tweak with this equality slightly. 

For any set $E \subset X$, and $\delta>0$ define the collections of "half-empty and half-full balls"
\begin{equation}\label{eq:-halffullcollection}
\cB(E,\delta) \defeq \{\clB=\clB(x,r) \subset X:  \Theta(E,\clB) \geq \delta\}.
\end{equation}
For consistency with our Theorem \ref{thm:inflation-mainthm}, we will use closed balls from here on out. Note that, as in Remark \ref{rmk:assumptions-mainthm}, we could also work with open balls by modifying the estimates with a factor. Our choices here are driven mostly with the need to be consistent, and to avoid switching between open and closed balls needlessly. 

The bridge between our main Theorem \ref{thm:inflation-mainthm} and the size of $\partial^* E$ is the following lemma. Recall the definition of the measure theoretic boundary from Equation \eqref{eq:mboundary}, and its quantitative version \eqref{eq:mboundaryeta}.

\begin{lemma}\label{lem:limsupmeastheor} Suppose that $(X,d,\mu)$ is $D$-measure doubling. For any $\delta>0, L\geq 1$ there is a $k\in \N$ so that we have
\[
\limsup L\cB(E,\delta) \subset \partial_{\delta D^{-k}}^* E.
\]
\end{lemma}
\begin{proof} Let $k\in \N$ be such that $4L \leq 2^k$.
Let $p \in \limsup L\cB(E,\delta)$. Then, there exists a sequence $(x_i)_{i\in \N}$ of points $x_i\in X$ and a sequence $(r_i)_{i\in \N}$ of positive real numbers with $\lim_{i\to\infty} r_i=0$ with $\clB(x_i,r_i) \in \cB(E,\delta)$ and $p\in \clB(x_i,Lr_i)$ for each $i\in\N$.

By definition of $\cB(E,\delta)$, we have $\Theta(E,\clB(x_i,r_i))\geq \delta$. That is $\mu(\clB(x_i,r_i) \cap E) \geq \delta \mu(\clB(x_i,r_i))$ and $\mu(\clB(x_i,r_i) \setminus E) \geq \delta \mu(\clB(x_i,r_i))$. Consider $s_i=2Lr_i$. Then $\clB(x_i,r_i) \subset \clB(p,s_i) \subset \clB(x,4Lr_i)$, and by doubling we get
$\mu(\clB(x_i,s_i) \cap E) \geq \delta \mu(\clB(x_i,r_i)) \geq D^{-k}\delta \mu(\clB(x_i,s_i))$ and $\mu(\clB(x_i,s_i) \setminus E) \geq \delta \mu(\clB(x_i,r_i)) \geq D^{-k}\delta \mu(\clB(x_i,s_i))$. 

Since $\lim_{i\to\infty} s_i = 0$, we have
\[
\limsup_{r \to 0} \frac{\mu(\clB(p,r)\cap E)}{\mu(B(p,r))}\geq \delta D^{-k} \text{ and } \limsup_{r \to 0} \frac{\mu(\clB(p,r)\setminus E)}{\mu(\clB(p,r))}\geq \delta D^{-k}.
\]
Thus, $p \in \partial^*_{D^{-k}\delta}(E)$ by definition.

\end{proof}

Next, we need a method to find points in $\bigcup \cB(E,\delta)$. The following lemma gives us this. It is a modified version of an argument that originally appeared in \cite[Proposition 3.9]{korte}.

\begin{lemma}\label{lem:curvelemma} Let $E \subset X$ and suppose that $\gamma:[0,1]\to X$ is a continuous curve. If for some $r>0$, we have
\[
\frac{\mu(\clB(\gamma(0),r)\cap E))}{\mu(\clB(\gamma(0),r))}\geq 2^{-1} \text{ and } \frac{\mu(\clB(\gamma(1),r)\cap E))}{\mu(\clB(\gamma(1),r))}\leq 2^{-1},
\] 
then, there exists a $t\in [0,1]$ so that $\gamma(t) \in \bigcup \cB(E,2^{-1}D^{-2})|_{2r}$.
\end{lemma}

\begin{proof}
Let $T=\{t \in [0,1] :\frac{\mu(\clB(\gamma(t),r)\cap E)}{\mu(\clB(\gamma(t),r))} \geq 2^{-1} \}$. We have $T \neq \emptyset$, since $0 \in T$. Let $t=\sup T$. We will argue that $t$ is the point we desire. There are a few cases to consider, which depend on the position of $t$.

\begin{enumerate}
\item Suppose $t=1$: We have $\frac{\mu(\clB(\gamma(t),r)\cap E)}{\mu(\clB(\gamma(t),r))}\leq 2^{-1}$ by assumption, and thus $\frac{\mu(\clB(\gamma(t),2r)\setminus E)}{\mu(\clB(\gamma(t),2r))}\geq 2^{-1}D^{-1}$. There is a $t'\in T$ so that $d(\gamma(t'),\gamma(t))\leq r$. 
 Doubling implies that $\mu(\clB(\gamma(t'),r))\geq D^{-2}\mu(\clB(\gamma(t),2r))$. Thus, 
\begin{equation}\label{eq:gammatbound}
\frac{\mu(\clB(\gamma(t),2r))\cap E)}{\mu(\clB(\gamma(t),2r))}\geq \frac{\mu(\clB(\gamma(t'),r)\cap E)}{D^2 \mu(\clB(\gamma(t'),r))}\geq 2^{-1}D^{-2}
.\end{equation}
Therefore, $\Theta(E,\clB(\gamma(t),2r)) \geq 2^{-1}D^{-2}$.
\item Suppose $t<1$: The argument is similar to the previous case, since we can choose a $t'\in T$ and a $t'' \not\in T$ with $d(\gamma(t'),\gamma(t))\leq r$ and $d(\gamma(t''),\gamma(t))\leq r$. With the same argument as above, we get Equation \eqref{eq:gammatbound}. On the other hand, by replacing $t'$ with $t''$ and $E$ with the set $X\setminus E$, we obtain 
\begin{equation}\label{eq:gammatpbound}
\frac{\mu(\clB(\gamma(t),2r)\setminus E))}{\mu(\clB(\gamma(t),2r))}\geq \frac{\mu(\clB(\gamma(t''),r)\setminus E)}{D^2 \mu(\clB(\gamma(t''),r))}\geq 2^{-1}D^{-2}
.\end{equation}
Now, estimates \eqref{eq:gammatbound} and \eqref{eq:gammatpbound} give $\Theta(E,\clB(\gamma(t),2r)) \geq 2^{-1}D^{-2}$.
\end{enumerate}
In both cases, $\clB(\gamma(t),2r) \in \cB(E,2^{-1}D^{-2})|_{2r}$ and $\gamma(t) \in \clB(\gamma(t),2r)$, which yields the claim.
\end{proof}

The second lemma concerns unions of the collection $\cB(E,\delta)$. This lemma is needed to verify the assumption of Theorem \ref{thm:inflation-mainthm}.

\begin{lemma}\label{lem:finiteperimunion} Suppose that $(X,d,\mu)$ is a PI-space. There exists constants $C \geq 1$ for which the following holds. For any $E\subset X$ and any ball $\clB=\clB(x,r)\subset X$, and any $s>0$. We have
\[
\cH^h_r\left(\bigcup \cB(E,2^{-1}D^{-2})|^{3\lambda \clB}_s\right) \geq C\Theta(E,\clB)h(\clB).
\]
\end{lemma}

\begin{proof}
Let

\begin{align*}
I_{r,E}&=\{x \in \clB : \inf_{r'\in (0,r)} \frac{\mu(E\cap \clB(x,r'))}{\mu(\clB(x,r'))} > 1/2\}, \text{ and } \\
O_{r,E}&=\{x \in \clB: \sup_{r'\in (0,r)} \frac{\mu(E\cap \clB(x,r'))}{\mu(\clB(x,r'))} < 1/2\}. \\
\end{align*}

By Lebesgue differentiation for a.e. $x\in E\cap \clB$, we have $\lim_{r\to 0} \frac{\mu(\clB(x,r)\cap E)}{\mu(\clB(x,r))}=1$, and for a.e. $x\not\in E \cap B$ we have $\lim_{r\to 0} \frac{\mu(\clB(x,r)\cap E)}{\mu(\clB(x,r))}=1$. From this, we have by the previous combined with Egorov's theorem that there is an $r_0\in (0,s/2)$ so that $\mu(I_{r_0,E}) > \frac{\Theta(E,\clB)}{2}\mu(\clB)$ and $\mu(O_{r_0,E}) > \frac{\Theta(E,\clB)}{2}\mu(\clB)$.

Choose a covering $\mathcal{C}$  of the set $\bigcup \mathcal{B}(E,2^{-1}D^{-2})|_{s}^{3\lambda \clB}$  with $\rad(\cC) \leq \rad(\clB)$. We will show that $\sum_{c\in \cC} h(c) \geq C\Theta(E,\clB)h(\clB)$ for a constant $C>0$ to be determined. There are a few cases to consider.  

First, consider the cases when
\[\mu\left(\bigcup 2\cC \cap I_{r_0,E}\right)\geq \mu(I_{r_0,E}\cap \clB)/2\]
or
\[
\mu\left(\bigcup 2\cC \cap O_{r_0,E}\right)\geq \mu(O_{r_0,E}\cap \clB)/2.
\]
Then $\mu(\bigcup 2\cC) \geq \min\{\mu(I_{r_0,E}\cap \clB),\mu(O_{r_0,E}\cap \clB)\}/2 \geq \Theta(E,\clB)\mu(\clB)/4$. But, then

\begin{align*}
\sum_{c\in \cC} h(c) &= \sum_{c\in \cC} \frac{\mu(\overline{c})}{\rad(\overline{c})} \\
&\geq \frac{1}{D\rad(\clB)}\sum_{c\in \cC} \mu(2c) \\
&\geq \frac{1}{D\rad(\clB)}\mu\left(\bigcup 2\cC\right) \geq \frac{\Theta(E,\clB)}{4D} h(\clB). 
\end{align*}
Thus, the desired bound holds with $C=(4D)^{-1}$.

Thus in the remaining case we can assume that

\begin{align*}
\mu\left(\bigcup 2\cC \cap I_{r_0,E}\right)\leq \mu(I_{r_0,E}\cap \clB)/2 \\
\mu\left(\bigcup 2\cC \cap O_{r_0,E}\right)\leq \mu(O_{r_0,E}\cap \clB)/2. 
\end{align*}

Let $g=\sum_{c\in \cC} \frac{1}{\rad(c)}1_{2c} + 1_{2\lambda \clB \setminus \lambda \clB}$, and let $E =  I_{r_0,E}\setminus \bigcup 2\cC, F = O_{r_0,E}\setminus \bigcup 2\cC$. Define 
\[
f(x) = \min\left(\inf_{\gamma:E\mapsto x} \int_\gamma g ds,1\right),
\]
where the infimum is taken over all rectifiable curves $\gamma:[0,1]\to X$ which connects a point in $\gamma(0) \in E$ to $\gamma(1)=x$. The infimum includes constant curves, and thus $f|_E=0$. 

On the other hand, we will show that $f|_F \geq 1$. Indeed,  let $\gamma$ be any rectifiable curve connecting $e \in E$ to $x\in F$. We will show that $\int_\gamma g ds\geq 1$, and by taking an infimum over $\gamma$, we obtain $f(x) = 1$.

First, if there is a $t\in [0,1]$ so that $\gamma(t) \not\in 2\lambda \clB$, then $\gamma$ contains a subsegment of length at least $\lambda \rad(\clB)$, which is contained in $2\lambda \clB \setminus \lambda \clB$. 
We have $g|_{2\lambda \clB \setminus \lambda \clB} \geq 1$, and thus $\int_\gamma g ds \geq 1$ as desired. 
Thus, assume that for all $t\in [0,1]$ we have $\gamma(t) \in 2\lambda \clB$.

 By Lemma \ref{lem:curvelemma} (applied with $r=r_0$) there exists a $t\in [0,1]$, so that $\gamma(t)\in \bigcup \cB(E,2^{-1}D^{-2})|_{2r_0} \subset \bigcup \cB(E,2^{-1}D^{-2})|_{s}$. By the previous paragraph, we have $\gamma(t) \in 2\lambda \clB$. 
 Further, since $2r_0 \leq \rad(\clB)\leq \lambda \rad(\clB)$, we also have $\gamma(t)\in \bigcup \cB(E,2^{-1}D^{-2})|^{3\lambda \clB}_{2r_0}$. Thus, $\gamma(t) \in \bigcup \cC$. Let $c\in C$ be such that $\gamma(t)\in c$ 

However, $\gamma(0), \gamma(1) \not\in \bigcup 2\cC$ by the assumption on $\gamma$. Thus $\gamma$ contains a subsegment of length at least $\rad(c)$ within $2c$. Since $g|_{2c}\geq \frac{1}{\rad(c)}$, we get $\int_{\gamma} g ds \geq 1$. This shows $f|_E = 1$.

 It is immediate to show that $f:X\to [0,1]$. Further, $g$ is an upper gradient of $f$ (i.e satisfied \eqref{eq:upgrad}). This is proven by a classical argument from e.g. \cite[Proposition 3.2]{jarvempaa}. Indeed, suppose that $x_1,x_2 \in X$, and that $\gamma:[0,1]\to X$ is a rectifiable curve with $\gamma(0)=x_1,\gamma(1)=x_2$. By symmetry, consider the case $f(x_1)<f(x_2)\leq 1$, and let $\gamma'$  be any curve connecting $E$ to $x_1$, and form $\gamma''$ by concatenating $\gamma'$ with $\gamma$. Then,
 
 \[
 f(x_2) \leq \inf_{\gamma_2: E  \mapsto x_2} \int_{\gamma_2} g \, ds \leq \int_{\gamma''} g \, ds \leq \int_{\gamma'} g \, ds + \int_{\gamma} g \, ds.
 \]
 Taking an infimum over $\gamma'$ yields $f(x_2) \leq \inf_{\gamma': E\mapsto x} g \, ds  + \int_{\gamma} g \, ds \leq f(x_1) + \int_{\gamma} g \, ds$. I.e. $|f(x_2)-f(x_1)|\leq \int_{\gamma} g \, ds$, as desired.
 
 By the the Poincar\'e inequality:
\begin{equation}\label{eq:pireminded}
\vint_{\clB} |f-f_\clB| d\mu \leq c_P \rad(\clB) \vint_{\lambda \clB} g d\mu.
\end{equation}
First, we estimate the right hand side from above:
\begin{equation}\label{eq:gbound}
c_P \rad(\clB)\vint_{\lambda \clB}gd\mu = \frac{c_P \rad(\clB)}{\mu(\lambda \clB)}\int \sum_{c\in C} \frac{1}{\rad(c)} 1_{2c} d\mu \leq \frac{c_P \rad(\clB)}{\mu(\lambda \clB)} \sum_{c\in C} \frac{\mu(2c)}{\rad(c)} \leq \frac{c_P D^L}{h(\clB)} \sum_{c\in C} h(c).
\end{equation}
In the last inequality, we used the doubling of $\mu$, and $L \in \N$ is such that $\lambda \leq 2^{L+2}$.

Thus, to obtain the desired lower bound for $\sum_{c\in \cC} h(c)$, we only need to show that $\vint_{\clB} |f-f_\clB| d\mu\geq 2^{-3}\Theta(E,\clB)$. This final claim we divide to two cases depending on $f_\clB$:
\begin{enumerate}
\item Suppose $f_\clB\leq 2^{-1}$: We have
$\vint_{\clB} |f-f_\clB| d\mu \geq \frac{1}{2}\frac{\mu(F)}{\mu(\clB)} \geq \frac{1}{8}\Theta(E,\clB)$, since $f|_F=1$.
\item Suppose $f_\clB> 2^{-1}$: We have
$\vint_{\clB} |f-f_\clB| d\mu \geq \frac{1}{2}\frac{\mu(E)}{\mu(\clB)} \geq \frac{1}{8}\Theta(E,\clB)$, since $f|_E=0$.
\end{enumerate}
\end{proof}

The following lemma verifies the conditions of Lemma \ref{lem:nice-collection} for finite perimeter sets.

\begin{lemma}\label{lem:finiteperim-nice} Let $X$ be a PI-space and $E \subset X$ a finite perimeter set. There is a constant $C>0$ so that if $\mathcal{B}' \subset 3\lambda \cB(E,2^{-1}D^{-2})$ is a disjoint collection, then
\[
\sum_{\clB\in \cB'} h(\clB) \leq C.
\]

\end{lemma}
\begin{proof}Assume that $E$ is a set of finite perimeter. Therefore, there exists a sequence $f_n \in L^1_{\rm loc}(X)$ and upper gradients $g_n \in L^1(X)$ with 
\[
\liminf_{n\to\infty} \int_X g_n \, d\mu < \infty.
\]

By passing to a subsequence, we can replace $\liminf$ by $\lim$. 
Let $C= \lim_{n\to\infty} \int_X g_n \, d\mu$. Fix $r>0$ and consider $\cB = \cB(E,2^{-1}D^{-2})|_{r}$. Assume that $\cB' \subset 3 \lambda \cB$ is a disjoint collection.  We will show that there is a constant $C'=C'(D,\lambda,C)$ with 
\[
\sum_{\clB\in \cB'} h(5\clB) \leq C'.
\]
Let $\cB''=\{\clB_1,\dots, \clB_K\} \subset \cB'$ be an arbitrary finite subcollection. It suffices to prove that $\sum_{B\in \cB''} h(5B) \leq C'$. 
For each $i=1,\dots, K$, we can choose a $N_i$ so that for $n\geq N_i$ we have $\int_{\frac{1}{3\lambda}\clB_i} |1_E-f_n|\,d\mu \leq 2^{-4}D^{-2}\mu(\clB_i)$. 

Since $\Theta(E,\frac{1}{3\lambda}\clB_i) \geq 2^{-1}D^{-2}$ for all $i=1\dots, K$, we get for each $n\geq N\defeq \max\{N_i : i=1,\dots, K\}$ that

\begin{align*}
\vint_{\frac{1}{3\lambda}\clB_i} |f_n-f_{n}|_{\clB_i}| d\mu &\geq \vint_{\frac{1}{3\lambda}\clB_i} |1_E-1_{E}|_{\clB_i}|-|f_n-1_E| - |1_E|_{\clB_i}-f_{n}|_{\clB_i}| d\mu \\
&\geq 2^{-1}D^{-2} - 2^{-3}D^{-2} \geq 2^{-2}D^{-2}.
\end{align*}
Therefore, by the Poincar\'e inequality, we get for each $i=1,\dots, K$ and any $n \geq N$ that
\[
2^{-2}D^{-4}\leq c_P \rad(\clB_i) \vint_{\clB_i} g_n  \,d\mu.
\]
 Reorganizing terms, and using doubling, we get $2^{-5}D^{-7}c_{P}^{-1} h(5\clB_i) \leq \int_{\lambda \clB_i} g_n \,d\mu$. Summing over $i$, and using doubling and the distjointness of $\clB_i$, we get a constant $L=c_P D^7 2^5$ so that 
\[
\sum_{\clB\in \cB''} h(5\clB) \leq L\int g_n \, d\mu 
\]
for all $n\geq N$. By setting $C'=2LC$ and sending $n\to\infty$ then the desired bound is obtained.
\end{proof}

Finally, we give a characterization of sets of finite perimeter using unions of sets. This characterization is essentially known (see e.g. \cite[Theorem 6.2]{korte} and \cite[Theorem 4.6]{KKST}), but our proof and statement are slightly different and we present a full argument in order to be self-contained.

\begin{theorem}\label{thm:set-finite-perim} Let $X$ be a PI-space. A subset $E$ is of finite perimeter, if and only if there are constants $C,L>0$ so that for every $r>0$ we have
\[
\cH^h_{L r}(\bigcup \cB(E,2^{-1}D^{-2})|_r) \leq C.
\]
\end{theorem}

\begin{proof} First prove the necessity. That is, assume that $E$ is of finite perimeter. Fix $r>0$ and consider $\cB = \cB(E,2^{-1}D^{-2})|_{r}$. Apply Lemma \ref{lem:5covering} to get a disjoint collection $\cB' \subset 3\lambda \cB$ with $\bigcup \cB \subset \bigcup 3\lambda \cB \subset 5\bigcup \cB'$.  By Lemma \ref{lem:finiteperim-nice} and doubling there exists a constant $C'$ with 
\[
\sum_{\clB\in \cB'} h(5\clB) \leq C',
\]
from which the claim that $\cH^h_{Lr}(\bigcup \cB(E,2^{-1}D^{-2})|_r)\leq DC'$ follows with $L=10\lambda$. (The extra factor of $2$ comes from the need to transition from closed to open balls here.)

Next to prove the sufficiently, assume that $E \subset X$ is measurable and \[
\cH^h_{5\lambda r}(\bigcup \cB(E,2^{-1}D^{-2})|_r) \leq C.
\]
 Fix a point $p\in X$.  Recall that a PI-space is measure doubling and thus separable.  Recall also that $\mu$ is a Borel measure on a complete and separable metric space, and thus inner regular. Therefore, since $X$ is doubling, by Egorov and Lebesgue differentiation Theorem, for each $n\in \N$, we can find compact subsets $E_n \subset E$ and $F_n \subset B(p,n) \setminus E$ with $\mu(E\setminus E_n) \leq 2^{-n}$ and $\mu(B(p,n) \setminus (E \cup F_n)) \leq 2^{-n}$, and for which there exist $r_n>0$ so that

\[
\inf_{r\in (0,r_n)} \frac{\mu(E \cap B(x,r_n))}{\mu(B(x,r_n))} > \frac{1}{2}
\] 
for all $x\in E_n$ and 
\[
\inf_{r\in (0,2r_n)} \frac{\mu( B(x,r_n) \setminus E)}{\mu(B(x,2r_n))} > \frac{1}{2}
\]
for all $x\in F_n$. It is clear that $F_n$ and $E_n$ are disjoint.

 Next, let $n\in \N$ be arbitrary. 
Let $\delta_n = \min(r_n, \inf_{a\in E_n,b \in F_n} d(a,b)(4L)^{-1})>0$. Now, choose a covering $\cC$ of $\bigcup \cB(E,2^{-1}D^{-2})|_{\delta_n}$ by open balls with $\rad(\cC) \leq L \delta_n \leq \inf_{a\in E_n,b \in F_n} d(a,b)/4$ and with $\sum_{c\in \cC}h(c) \leq 2C$.

Define $g_n=\sum_{c\in \cC} \frac{1}{\rad(c)}1_{2c}$, and $f_n(x)=\min(1,\inf_{\gamma:F_n \mapsto x} \int_\gamma g \,ds)$. 
Just as in Lemma \ref{lem:finiteperimunion}, we get $f_n |_{E_n}=1$ and $f|_{F_n}=0$ since by Lemma \ref{lem:curvelemma} for any $\gamma$ which connects $E_n$ to $F_n$ we have $\gamma \cap \bigcup \cB(E,2^{-1}D^{-2})|_{\delta_n} \neq \emptyset$ and $\diam(\gamma) \geq 2\rad(\cC)$. 
Further, $f_n:X\to [0,1]$. As in the proof of  Lemma \ref{lem:finiteperimunion}, $g_n$ is an upper gradient of $f_n$, and we have

\[
\liminf_{n\to\infty} \int_X g_n \, d\mu \lesssim \liminf_{n\to\infty} \sum_{c \in \cC} h(c) <\infty.
\]

We are left to show that for every $R>0$ we have $\lim_{n\to \infty} \int_{B(p,R)}|1_E-f_n|\, d\mu = 0$. 
For $n\geq R$ we have $B(p,R) \subset E \cup ( B(p,n) \setminus E)$ and $1_E(x)=f_n(x)$ when $x\in E_n \cup F_n \cap B(p,R)$. Thus, by the choice of $E_n$ and $F_n$, we have
\[
\int_{B(p,R)}|1_E-f_n|\, d\mu \leq \mu(E \setminus E_n) + \mu(B(p,n)\setminus (E \cup F_n)) \leq 2^{1-n},\]
and the claim follows.

\end{proof}

\subsection{Strong isoperimetric inequality}

Using the previous section, we obtain the strong isoperimetric inequality.
\begin{theorem}\label{thm:isoperimetric-ineq}
Suppose that $X$ is a complete PI-space. 
There is a constant $C$ so that for any closed ball $\clB$ and any measurable set $E \subset \clB$ we have
\[\mathcal{H}^h_{\rad(\clB_0)}(\partial^* E \cap 3\lambda \clB) \geq \frac{1}{C} \Theta(E,\clB) h(\clB).\]
Indeed, if $\lambda \leq 2^L$, then
\[\mathcal{H}^h_{\rad(\clB_0)}(\partial^*_{2^{-1}D^{-4-L}} E \cap 3\lambda \clB) \geq \frac{1}{C} \Theta(E,\clB) h(\clB).\]
\end{theorem}

\begin{proof} Without loss of generality, assume that $\Theta(E,\clB)>0$. 

Let $\cB = \cB(E,2^{-1}D^{-2}) \cup \{B\}$. Then, for any $\clB_0 \in \cB$, we have $\Theta(E,\clB_0) \geq \min(D^{-2},\Theta(E,\clB))$. By Lemma \ref{lem:finiteperimunion}, we have a constant so that for any $s\geq 0$ we have
\[
\cH^h_{\rad(\clB_0)}(\cB|_s^{3\lambda \clB_0}) \geq C^{-1}\min(D^{-2},\Theta(E,\clB))h(\clB_0).
\]

Now, let $\cB'=3\lambda \clB$, for which we also get, for any $\clB_0 \in \cB'$
\[
\cH^h_{\rad(\clB_0)}(\cB'|_s^{3\lambda \clB_0}) \geq C^{-1}\min(D^{-2},\Theta(E,\clB))h\left(\frac{1}{3\lambda}\clB_0\right).
\]
Thus, by Theorem \ref{thm:inflation-mainthm} and doubling we get (after adjusting the constant $C$), that
\[
\cH^h_{\rad(\clB_0)}(\limsup \cB' \cap \cB_0) \geq C^{-1}\min(D^{-2},\Theta(E,\clB))h(\clB_0),
\]
for any $\clB_0 \in \cB'$. By Lemma \ref{lem:limsupmeastheor}, we have $\limsup \cB' \subset \partial^*_{2^{-1}D^{-4-L}} E \subset \partial^* E$, where $L\in \N$ is such that $\lambda \leq 2^L$, and the claim follows. Further, $\Theta(E,B) \leq 1$, so $D^{-2} \geq D^{-2}\Theta(E,\overline{B})$, and we get 
\[
\cH^h_{\rad(\clB_0)}(\limsup \cB' \cap \cB_0) \geq C^{-1}D^{-2}\Theta(E,\clB)h(\clB_0).
\]
\end{proof}

\begin{remark} The constants in the previous theorem could be improved when $X$ is geodesic. In that case, $\lambda=1$ is possible in the Poincar\'e inequality \eqref{eq:PIconstant}, see \cite{hajkos}, which allows us to take $L=1$. In the geodesic setting, also the constant $2D^{-2}$ in Lemma \ref{lem:curvelemma} could be improved to unity -- see the argument in \cite[Proposition 3.9]{korte}. With these two, we could replace $\partial^* E$ with $\partial^*_{D^{-2}} E$ in the statement. 
\end{remark}

\subsection{Federer Characterization}

We will next prove Theorem \ref{thm:federer}.

\begin{proof}[Proof of Theorem \ref{thm:federer}]

Suppose first that $\cH^h(\partial^* E)<\infty$. 
We verify that there is a constant $C>1$ for which $\cH^h_{15\lambda r}(\bigcup \cB(E,2^{-1}D^{-2} )|_r) \leq C$ for all $r>0$. 
Once this has been proven, Theorem \ref{thm:set-finite-perim} implies that $E$ is a set of finite perimeter.
 
Now, let $\cB' \subset 3\lambda \cB(E,2^{-1}D^{-2})|_r$ be a disjoint collection so that $\bigcup \cB(E,2^{-1}D^{-2})|_r \subset \bigcup 5 \cB'$. 
Since $1/(3\lambda) \cB' \subset \cB(E,2^{-1}D^{-2})$, by the strong isoperimetric inequality Theorem \ref{thm:isoperimetric-ineq}, we have
\[
\cH^h(\partial^*E \cap \clB) \geq \cH^h_{\rad(\clB)}(\partial^* E \cap \clB)  \geq \frac{1}{2 D^2 C} h\left(\frac{1}{3\lambda}\clB\right), 
\]
for every $\clB \in \cB'$. Summing over $\clB \in \cB'$ together with doubling gives
\[
\sum_{\clB\in \cB'} h(5\clB) \lesssim \sum_{\clB\in \cB'}h\left(\frac{1}{3\lambda}\clB\right) \lesssim \sum_{\clB\in \cB'} \cH^h(\partial^*E\cap \clB) <  \cH^h(\partial^*E) <\infty,
\]
where in the second-to-last inequality we used disjointness of $\cB'$. Since $\bigcup \cB(E,2^{-1}D^{-2}) \subset \bigcup 5 \cB'$, the desired content bound follows.

For the other direction of the theorem, assume that $E$ is a set of finite perimeter. If $E$ is of finite perimeter, then $\cH^h(\partial^* E)<\infty$ follows from \cite[Theorem 5.5.]{Adoubling}. Indeed, this is the only instance in which we need to refer to a result about finite perimeter sets, which we do not prove.  The proof of this uses quite different techniques than our arguments here.
\end{proof}

\subsection{Generic curves and sets of finite perimeter}

In this subsection, we discuss curves going from the measure theoretic interior $I_E$ to the measure theoretic exterior $O_E$. By Lemma \ref{lem:curvelemma}, each curve $\gamma \in \Gamma_{I_E,O_E}$ will intersect $\bigcup \cB(E,2^{-1}D^{-2})$. However, this need not imply that $\gamma$ must intersect $\limsup \cB(E,2^{-1}D^{-2})$. However, when $E$ is a set of finite perimeter, we show that for $\Mod_1$-a.e. curve this is the case. First, we need a simple lemma relating the $h$-content to modulus. 

For a set $E \subset X$ and $\delta>0$, denote by $\Gamma_{E,\delta}$ the collection of non-constant rectifiable curves which intersect $E$ and have $\diam(\gamma)\geq \delta$, where $\diam(\gamma) = \sup_{t,s\in [0,1]} d(\gamma(s),\gamma(t))$. It will also be convenient to identify the curve $\gamma$ with its image $\{\gamma(t): t\in [0,1]\}$. With this convention, we can write $\gamma \cap A\neq\emptyset$, when $A\subset X$, to mean that there is a $t\in [0,1]$ so that $\gamma(t) \in A$.

\begin{lemma} \label{lem:modcontent} Let $X$ be a PI-space, and $E \subset X$. Then,  $\Mod_1(\Gamma_{E,\delta}) \leq D \cH^h_{\delta/3}(E)$.
\end{lemma}
\begin{proof} Without loss of generality, assume that $\cH^h_{\delta/3}(E)<\infty$. Then 
for any $\epsilon>0$ we can find a collection $\cC$ of open balls with $\rad(\cC) \leq \delta/3$ and which covers $E$ and with $\sum_{c\in \cC} h(c) \leq \cH^h_{\delta/3}(E) + \epsilon/D$. Let $g=\sum_{c\in \cC} \frac{1}{\rad(c)} 1_{2c}$. By the same calculation as in Lemma \ref{lem:finiteperimunion}, we get
\[
\int g \, d\mu \leq D \sum_{c\in \cC} h(c) \leq D\cH^h_{\delta/3}(E) + \epsilon.
\]
Next, we will show that $\int_\gamma g \,ds \geq 1$ for any $\gamma \in \Gamma_{E,\delta}$. From which $\Mod_1(\Gamma_{E,\delta})\leq D\cH^h_{\delta/2}(E) + \epsilon$ follows. The claim follows from this since $\epsilon>0$ was arbitrary.

Let $\gamma \in \Gamma_{E,\delta}$ be arbitrary. Since $\gamma$ intersects $E$, and since $\cC$ is a cover of $E$, there exists a $c\in \cC$ so that $\gamma$ intersects $c$. Since $\diam(\gamma) \geq 2\rad(c)$, we have that $\gamma$ is not contained in $2c$. Thus, it contains a sub-segment of length at least $\rad(c)$ in the annulus $2c\setminus c$. Thus, $\int_\gamma g \, ds \geq \rad(c) \frac{1}{\rad(c)} \geq 1$, as claimed.
\end{proof}

With this lemma, we can prove the final statement of the paper.

\begin{proof}[Proof of Theorem \ref{thm:generic-curves}] Let $\Gamma_b =  \{\gamma \in \Gamma_{I_E,O_E}, \gamma \cap \partial^* E = \emptyset\}$ be the collection of curves in $\Gamma_{I_E,O_E}$ which do not pass through the measure theoretic boundary.  For each $n\in \N$ define the collection $\Gamma_n = \{\gamma : \gamma \in \Gamma_b, \diam(\gamma) \geq \frac{1}{n} \}.$ We have $\Gamma_b = \bigcup_{n\in \N} \Gamma_n$. Since $\Mod_1$ is subadditive, the claim that $\Mod_1(\Gamma_b)=0$ follows once we show that $\Mod_1(\Gamma_n)=0$ for each $n\in \N$. Thus, fix $n\in \N$ in what follows and let $\delta = \frac{1}{4n}$.

 Let $\cB=3\lambda \cB(E,2^{-1}D^{-2})$.
 Let $\cB'  \subset \cB$ be any disjoint collection.  By Lemma \ref{lem:finiteperim-nice} we have a value $C$ (which does not depend on $\delta$) so that $\sum_{\clB \in \cB'} h(\clB) \leq C <\infty $. Thus $\cB$ verifies the assumptions of Lemma \ref{lem:nice-collection}.

Next, fix $\epsilon>0.$ Let $r_1=\delta / (10C)$, where $C\geq 1$ is given by Lemma \ref{lem:nice-collection}. By Lemma \ref{lem:nice-collection} we can choose a finite collections $\cB_1 \subset \cB|_\delta$ so that

\[\mathcal{H}^h_{\delta}(\bigcup \cB|_{r_1} \setminus 5\mathcal{B}_1) \leq 2^{-1}D^{-1} \epsilon.\]

Proceed recursively to define a decreasing sequence $r_k$ and collections $\cB_k \subset \cB|_{r_k}$. Suppose that $r_k$ is defined. Then, define $r_{k+1} \leq \min_{B\in\mathcal{B}_k} \rad(B)/2$. Again, using Lemma \ref{lem:nice-collection} choose a finite collection $\mathcal{B}_{k+1}\subset \mathcal{B}|_{r_{k+1}}$ so that

\[\mathcal{H}^h_{\delta}(\bigcup \mathcal{B}|_{r_{k+1}} \setminus 5 \mathcal{B}_{k+1}) \leq D^{-1}2^{-k-1}\epsilon.\]

Let $\Gamma_{n,k}=\{\gamma \in \Gamma_{n} : \gamma \cap (\bigcup \mathcal{B}|_{r_{k}}) \setminus 5 \mathcal{B}_{k} \neq \emptyset \}$. By Lemma \ref{lem:modcontent}, we have  $\Mod_1(\Gamma_{n,k}) \leq \epsilon 2^{-k}$. By subadditivity we get $\Mod_1(\bigcup_{k \in \N} \Gamma_{n,k}) \leq \epsilon$. The claim then follows from arbitrariness of $\epsilon>0$ once we show that $\Gamma_n \subset \bigcup_{k\in \N} \Gamma_{n,k}$. 

In other words, we prove that any curve $\gamma \in \Gamma_n$ must intersect one of the sets $(\bigcup \mathcal{B}|_{r_{k}}) \setminus 5 \mathcal{B}_{k}$, for some $k\in \N$. 
This follows by showing that for any $\gamma  \in \Gamma_{I_E,O_E} \setminus \bigcup_{k\in \N} \Gamma_{n,k}$ with $\diam(\gamma) \geq \frac{1}{n}$ we have $\gamma \cap \partial^* E \neq \emptyset$, that is $\gamma \not\in \Gamma_n$.

Let $\gamma \in \Gamma_{I_E,O_E} \setminus \bigcup_{k\in \N} \Gamma_{n,k}$ with $\diam(\gamma) \geq \frac{1}{n}$ be arbitrary.   Define sets by $\mathcal{K}_k = \bigcup 5\cB_k \cap \gamma$. The set $\mathcal{K}_k$, for $k\in \N$, is compact since $\cB_k$ is a finite collections of closed balls and  the image of $\gamma$ is compact. Suppose that $\cap_{k\in\N} \mathcal{K}_k \neq \emptyset$.  If this is the case, there exists a $t\in [0,1]$ for which $\gamma(t) \in \cap_{k\in \N} \mathcal{K}_k$. 
Therefore, there is a sequence of balls $(\clB_i)_{i\in \N}$ with $\clB_i \in \cB_i$ and $\gamma(t)\in 5\cB_i$. 
From this we get $\gamma(t) \in \limsup 15\lambda \cB(E,2^{-1}D^{-2})$, and $\gamma(t)\in \partial^* E$ by Lemma \ref{lem:curvelemma}. In particular, $\gamma \cap \partial^* E \neq \emptyset$ and $\gamma \not\in \Gamma_n$ as desired.

What remains is to derive a contradiction from the assumption $\cap_{k\in\N} \mathcal{K}_k = \emptyset$. Suppose this is the case. Then, by the finite intersection property, there exists an integer $N\in \N$ so that $\cap_{k=1}^N \mathcal{K}_k = \emptyset$. By Lemma \ref{lem:curvelemma}, since $\gamma \in \Gamma_{I_E,O_E}$, we have that $\gamma \cap \bigcup \cB|_{r_N} \neq \emptyset$. 
Then, there exists a $t\in [0,1]$ with $\gamma(t) \in \bigcup \cB|_{r_N}$. Since $\gamma(t) \not\in \cap_{k=1}^N \mathcal{K}_k$, there must exist a $k\in [1,N]\cap \N$ with $\gamma(t) \in \cB|_{r_N} \setminus \bigcup 5\cB_k \subset \cB|_{r_k} \setminus \bigcup 5\cB_k$. Then $\gamma \in \Gamma_{n,k}$, by definition. Thus, $\Gamma_n \subset \bigcup_{k\in\N} \Gamma_{n,k}$.

\end{proof}

\bibliographystyle{acm}
\bibliography{pmodulus}
\def\cprime{$'$}

\end{document}